\documentclass[onefignum,onetabnum]{siamart190516}
\usepackage{amssymb}
\usepackage{eucal}
\newcommand{\R}{{\Bbb R}}

\usepackage[normalem]{ulem}
\usepackage{physics}
\usepackage{amsmath}
\usepackage{tikz}
\usepackage{mathdots}
\usepackage{yhmath}
\usepackage{cancel}
\usepackage{color}
\usepackage{siunitx}
\usepackage{array}
\usepackage{multirow}
\usepackage{gensymb}
\usepackage{tabularx}
\usepackage{extarrows}
\usepackage{booktabs}
\usetikzlibrary{fadings}
\usetikzlibrary{patterns}
\usetikzlibrary{shadows.blur}
\usetikzlibrary{shapes}
\usepackage[caption=false]{subfig}
\usepackage{lipsum}
\usepackage{amsfonts}
\usepackage{graphicx}
\usepackage{epstopdf}
\usepackage{algorithmic}
\ifpdf
  \DeclareGraphicsExtensions{.eps,.pdf,.png,.jpg}
\else
  \DeclareGraphicsExtensions{.eps}
\fi

\newsiamremark{remark}{Remark}
\newsiamremark{hypothesis}{Hypothesis}
\crefname{hypothesis}{Hypothesis}{Hypotheses}
\newsiamthm{claim}{Claim}

\headers{Bistable waves in Belousov-Zhabotinsky reaction}{K. Has\'ik, J. Kopfov\'a,  P. N\'ab\v{e}lkov\'a, and S. Trofimchuk}

\title{Bistable wavefronts in the delayed Belousov-Zhabotinsky reaction\thanks{\funding{The work of Karel Has\'ik, Jana Kopfov\'a and Petra N\'ab\v{e}lkov\'a was supported  by the institutional support
for the development of research organizations I\v{C}O 47813059. Sergei Trofimchuk  was  also partially  supported by FONDECYT (Chile),   project 1231169. }}}

\author{Karel Has\'ik\thanks{Mathematical Institute, Silesian University, 746 01 Opava, Czech Republic
  (\email{karel.hasik@math.slu.cz}, \email{jana.kopfova@math.slu.cz}, \email{petra.nabelkova@math.slu.cz}).}
\and Jana Kopfov\'a \footnotemark[2]
\and Petra N\'ab\v{e}lkov\'a \footnotemark[2]
\and Sergei Trofimchuk \thanks{Instituto de Matem\'aticas, Universidad de Talca, Casilla 747,
Talca, Chile, corresponding author
  (\email{trofimch@inst-mat.utalca.cl}).}}

\usepackage{amsopn}

\ifpdf
\hypersetup{
  pdftitle={Bistable waves in Belousov-Zhabotinsky reaction with delay},
  pdfauthor={Karel Has\'ik, Jana Kopfov\'a,  Petra N\'ab\v{e}lkov\'a and Sergei Trofimchuk}}
\fi

\begin{document}

\maketitle

% REQUIRED
\begin{abstract}
 We study the Murray adaptation of the  Noyes-Field  five-step model of the Belousov-Zhabotinsky (BZ) reaction in the case when a tuning parameter $r$,  
which determines the level of the bromide ion far ahead of the propagating wave, is bigger than 1 and when the delay in generation of the bromous acid is taken into account. The existence of wavefronts in the delayed BZ system was previously established only in the monostable situation with $r \in (0,1]$, the physically relevant bistable situation where $r >1$  (in real experiments $r$ varies between 5 and 50) was left open.  We complete the study by  showing that the BZ system with $r >1$ admits monotone traveling fronts. Note that one of the stable equilibria of the BZ model is not isolated. This circumstance  does not allow the direct application of  the topological or analytical methods previously elaborated for  the analysis of the   existence of bistable waves. 
 \end{abstract}

% REQUIRED
\begin{keywords}
  Bistable model, delay, traveling front, Belousov-Zhabotinsky reaction
\end{keywords}

% REQUIRED
\begin{AMS}
   35C07, 35R10, 35K57
  \end{AMS}

\section{Introduction} Adapting  the Noyes-Field  five-step model \cite{MurV1} 
to describe  the propagation of wave bands in the Belousov-Zhabotinsky (BZ for short)  reaction, J.D. Murray in  \cite{Mur1,Mur2} proposed the reaction-diffusion system 
\begin{equation}\label{1}
\begin{array}{ll}
     u_t(t,x) = 
     D_u\Delta u(t,x)  + u(t,x)(1-u(t,x)-w(t-h,x)),
    &    \\
     w_t(t,x) = D_w\Delta w(t,x)  -b u(t,x)w(t,x), \ u, w \geq 0, \ x \in \R^2,& 
\end{array}%
\end{equation}
(with $h=0$) to describe planar waves $(u,w)=(\phi,\theta)(\nu\cdot x+ct)$, $(\phi,\theta)(-\infty)= (0,r),$ 
$(\phi,\theta)(+\infty)= (1,0),$ $r, b, D_u, D_w \geq 0,  \|\nu\|=1$, $0 \leq u \leq 1, \ 0 \leq w \leq r, $ propagating 
 in a thin layer of reactant solution filled in a Petri plate. 
The variables $u, w$ are dimensionless and proportional to the bromous acid and bromide ion concentrations, respectively. See \cite{gibbs,Mur1,Mur2,NIU} for the derivation of the model (\ref{1}) (without delay) and the chemistry around it. A positive delay $h$ was later incorporated in 
(\ref{1}) by J.~Wu and X.~Zou \cite{wz}, paper \cite{TPTb} provides additional motives for this inclusion. 

In order to visualise  the quasi-monotonicity properties of (\ref{1}), it is convenient to introduce a new variable $v = 1-w$.  Then, 
assuming that $D_u=D_w=1$ and  that the level $r$ of concentration $w$ of the bromide ion far ahead of the propagating wave
is known and fixed (in the real experiments $r$ can vary from 5 to 50, $b$ can take values from 2.5 to 12.5 \cite{Mur1}), we  transform (\ref{1}) into the system 
\begin{equation}\label{1r}
\begin{array}{ll}
     u_t(t,x) = \Delta u(t,x)  + u(t,x)(1-r-u(t,x)+rv(t-h,x)),
    &    \\
     v_t(t,x) = \Delta v(t,x)  +b u(t,x)(1-v(t,x)), \ u, v \in [0,1], \  \ x \in \R^2,& 
\end{array}%
\end{equation}
and look for the non-negative wave solutions $(u,v)(t,x)=(\phi,\psi)(\nu\cdot x+ct)$ satisfying the normalised boundary conditions 
$(\phi,\psi)(-\infty)= (0,0), 
(\phi,\psi)(+\infty)= (1,1)$ and  $0 \leq \phi \leq 1, \ 0 \leq \psi \leq 1$.  It is known \cite{MF,TPTa} that the profiles  $\phi(s), \psi(s)$ of such wavefronts (if they exist) are strictly monotone functions: in fact, 
 $\phi'(s), \psi'(s)>0$ for all $s \in \R$. 
Importantly, the value of $r$  determines the mathematical nature of the  BZ system. Namely, if $r \in (0,1)$, 
 the system  (\ref{1}) is monostable non-degenerate and  is  monostable degenerate when $r=1$; on the other hand, with $r >1$  
 the BZ system  is bistable \cite{TPTa}. Note that the standard mono- and bi-stability definitions \cite[p. 158]{Volp} require  the stability analysis of steady states of the following %bidimensional 
 system 
 \begin{equation}\label{1hr}
\begin{array}{ll}
    \phi'(t) = \phi(t)(1-r-\phi(t)+r\psi(t-h)),
    &    \\
    \psi'(t) = b\phi(t)(1-\psi(t)), \ (\psi, \phi) \in [0,1]\times [0,1].& 
\end{array}%
\end{equation}
describing spatially  homogeneous solutions for (\ref{1r}).  It is easy to see that the spectra of linearizations of (\ref{1hr}) at $(0,0)$ and $(1,1)$ are $\{0,1-r\}$ and $\{-1,-b\}$, respectively. Thus the equilibrium $(1,1)$ is  exponentially stable independently on the size of the positive parameters $h, b, r$.  
Since the equilibrium $(0, 0)$ is not isolated, its analysis is more delicate: see Appendix \ref{ApA} where the (non-asymptotical) stability of solution $(0,0)$ for (\ref{1hr}) with $r >1$  is proven.

The local and non-delayed case ($h=0$) is simpler to study since it allows the plane phase analysis to find the bistable waves. This technically substantial work was realised  by Y. Kanel in \cite{Ka2}. It was  preceded by the abundant numerical evidences for the presence of bistable fronts in the non-delayed  BZ system  \cite{man, Mur1,Mur2, QUI} and by the theoretical work of R.G. Gibbs \cite{gibbs}. In the latter paper,  on the base  of  geometrical-analytical arguments,   the existence [respectively, non-existence] of bistable waves for (\ref{1}) admitting only diffusion in bromous  acid, $D_u=1, D_w=0$ [respectively, admitting only diffusion in bromide ion, $D_u=0, D_w=1$] was established. It is  also worth to mention here  the recent study by H.-T. Niu {\it et al} \cite{NIU} where the existence of V-shaped traveling fronts in (\ref{1}) with $r >1, h=0$ was proved by constructing suitable super- and  sub-solutions.

Now, the existence, stability and uniqueness of wavefronts in the BZ system with spatiotemporal interactions (which includes (\ref{1}) as a particular case) was previously analysed  only in the monostable situation $r \in (0,1]$ under different conditions imposed on the parameters $c, h, r, b$ or interaction kernels, see \cite{bn,DQ,HAN2022103423,MF,ma,TPTa,TPTb,Troy,Volp,LinaWang,wz,Zh} and the references therein. The chemically relevant bistable case $r >1$  (recall that in real experiments $r$ varies between 5 and 50) was left open.  
We believe that  one of the main reasons of such a  circumstance is that the stable equilibrium $(0,0)$ of the quasi-monotone BZ model (\ref{1r}) is not isolated. This does not allow the direct application of  topological or analytical methods previously elaborated for the  analysis of  the existence of bistable waves in the monotone non-local or delayed systems (see \cite{FW,V20} for more references). 
In this paper, we propose a different approach  to complete the study  by showing that the delayed BZ system with $r >1$ admits monotone traveling fronts. As a by-product,  we obtain an alternative proof for the above mentioned Kanel result \cite[Theorem 4]{Ka2} concerning the existence of bistable fronts in the non-delayed BZ model.  Our main conclusion  is the  following: 

\begin{theorem} \label{mainT} For every $r>1, b>0, h \geq 0,$ there exists a unique $c_\star=c_\star(b,r,h)$ for which system (\ref{1r}) admits a wavefront  $(u,v)(t,x)=(\phi,\psi)(\nu\cdot x+c_\star t)$.  Furthermore,   $\phi'(t) >0,$  $\psi'(t)>0$, $\phi(t) < \psi(t)$ for all $t \in \R$  and  $c_\star \in (0,2)$. 
\end{theorem}
It should be noted that the last three inequalities in the above statement and the properties of the uniqueness and positivity of $c_\star$ were already proven in \cite{TPTa}. Nevertheless,  we include this information in the theorem to draw more complete picture of bistable wavefronts in the BZ system. 

Consequently,  Theorem \ref{mainT} will be proved if we establish the existence of positive wavefronts. We will reach this goal  by  approximating (\ref{1r})  with  the perturbed system \eqref{2rr}, see Section 2, parametrised by a small positive parameter $\epsilon$. We show that for each small $\epsilon >0$  perturbed equations  have a monotone bistable wavefront $(u,v)= (\phi_\epsilon, \psi_\epsilon)$ traveling with a positive speed $c_*(\epsilon)$. The existence of such a  wave is obtained by applying Fang and Zhao general theory  of bistable waves for  monotone  evolution systems \cite{FW}  to our particular situation. By itself, this application is not  trivial and needs  verifications of several  fundamental hypotheses. The most complicated of them, in view of the bidimensional character of the system, is  the counter-propagation property \cite{FW}. It can not be obtained  as  in the one-dimensional case, \cite{FW} and requires some additional ideas. We prove the counter-propagation by  constructing   appropriate sub-solutions and super-solutions in Appendices C and D. We hope that our approach can be useful in other similar situations. 
Next, in Section 3, we prove the existence of traveling waves by obtaining uniform a priori estimates for $c_\star$ and  taking the limit of $(\phi_\epsilon, \psi_\epsilon)$ when $\epsilon \to 0^+$.  Due to the degeneration of the equilibrium $(0,0)$, the latter task can be considered as the most difficult  in  the whole proof. We overcome this obstacle  by tracking the evolution of  $(\phi_\epsilon, \psi_\epsilon)$ as $\epsilon \to 0^+$ on the two-dimensional unstable manifold $\Gamma(\epsilon, c_\epsilon), \ \epsilon >0,$ of the equilibrium $(0,0)$. The existence and properties of $\Gamma(\epsilon, c_\epsilon)$ are given  in Theorem \ref{invman}. 
The paper is concluded with 4 Appendices. We have already mentioned two of them, C and D.  In addition, in Appendix  A the bistable character of  equation (\ref{1hr}) is shown followed by the detailed analysis of the perturbed system and accompanied with phase space analysis numerical pictures. In Appendix B we are checking another fundamental hypothesis (the unordering property) for the theory  in \cite{FW}. Again, the verification of this hypothesis is  more difficult than in one-dimensional case due to the reducibility of matrices obtained from the linearisation of the perturbed system at its stable equilibria. See Subsection 2.1 for more details.

\section{Monotone wavefronts for a perturbed BZ system}
\label{sec:monotone}

\subsection{Perturbed BZ system: existence of monotone wavefronts} \hfill \\

In order to prove Theorem \ref{mainT}, we have to solve  the  boundary value problem  which consists   from the differential equations  for the wave profiles $\phi, \psi$:  
 \begin{equation}\label{b1hr}
\begin{array}{ll}
    \phi''(t) -c\phi'(t)+ \phi(t)(1-r-\phi(t)+r\psi(t-ch))=0,
    &    \\
    \psi''(t) -c\psi'(t) +b\phi(t)(1-\psi(t))=0, & 
\end{array}%
\end{equation}
and the boundary conditions    $(\phi, \psi)(-\infty) =(0,0)$, $(\phi, \psi)(+\infty) =(1,1)$.

 Solutions of the above  problem will be obtained as limits (when  $\epsilon \to +0$) of the positive wavefronts to   the perturbed non-degenerate system 
 \begin{equation}\label{difference1hr}
\begin{array}{ll}
    \phi''(t) -c\phi'(t)+ \phi(t)(1-r-\phi(t)+r\psi(t-ch))=0,
    &    \\
    \psi''(t) -c\psi'(t)+(b\phi(t)-\epsilon \psi(t))(1-\psi(t))=0.& 
\end{array}%
\end{equation}
Note that the whole evolution system is given by the system of two reaction-diffusion equations (with $\epsilon >0$)
\begin{equation}\label{2rr}
\begin{array}{ll}
     u_t(t,x) =  u_{xx}(t,x)  + u(t,x)(1-r-u(t,x)+rv(t-h,x)),
    &    \\
     v_t(t,x) = v_{xx}(t,x)  + (b u(t,x)- \epsilon v(t,x))(1-v(t,x)), \ u, v \geq 0, \ x \in \R.& 
\end{array}%
\end{equation}
The proof of existence of  wavefronts for  (\ref{2rr})  is our main goal in this subsection: 
\begin{theorem}\label{2.1}  For every $b, h >0, r >1$ and sufficiently  small positive $\epsilon$,  the system (\ref{2rr}) has a monotone bistable wavefront 
 $u=\phi_\epsilon(x+c_*(\epsilon)t), \ v=\psi_\epsilon(x+c_*(\epsilon)t)$. 
\end{theorem}
\begin{proof} We use the theory of abstract monotone bistable  evolution systems developed in \cite{FW}. 
Thus we need to recall some definitions, notations and results from  \cite{FW}.  In particular, $\mathcal X^+$  stands for the subset $C([-h,0], [0,1]^2)$  of  the Banach lattice  $C([-h,0], \R^2)$ of continuous functions  with the uniform norm and standard  ordering of the elements. The standard order notation here is 
$$(\phi_2, \psi_2) \geq (\phi_1, \psi_1) \quad \mbox{if}\quad  \phi_2(s) \geq \phi_1(s), \  \psi_2(s) \geq \psi_1(s), \ s \in [-h,0],$$
$$(\phi_2, \psi_2) \gg (\phi_1, \psi_1)  \quad \mbox{if}\quad  \phi_2(s) > \phi_1(s),\  \psi_2(s) > \psi_1(s), \ s \in [-h,0].$$ 
Hence, if we set $\mathcal C: =  C([-h,0]\times \R, [0,1]^2)$, then  $\mathcal C$ can be regarded as the space of all continuous functions $\gamma: \R \to \mathcal X^+$.  Using the latter interpretation of $\mathcal C$,  we can equip this  space with the metrizable topology of uniform convergence on compact subsets of $\R$. If $\gamma \in \mathcal C$ is a constant function we can identify it with an element $\gamma \in \mathcal X^+$.  Let   $\mathcal C_\gamma$ and $[O,\gamma]_\mathcal C$ [respectively, $\mathcal X_\gamma$] denote the set of all elements 
 $\zeta \in \mathcal C$ [respectively, $\xi \in \mathcal X^+$] 
such that 
$0 \leq \zeta(t) \leq \gamma(t)$ for all $t \in \R$ [respectively, $0 \leq \xi \leq \gamma$].  

Note that for each pair of non-negative continuous and bounded initial data $u(s,x)=u_0(s,x) \in [0,1],$ $v(s,x) = v_0(s,x)\in [0,1],$ $(s,x) \in [-h,0]\times \R$,  system \eqref{2rr} has a unique continuous non-negative mild solution $(u,v)= (u(t,x),v(t,x)) \in [0,1]^2,  x \in \R,$ defined for all $t \geq 0$. Here the mild solutions are defined by means of the integral equations 
$$
u(t,x) = (S(t)u_0)(x) + \int_0^tS(t-s)(u(s,\cdot)(M+1-r-u(s,\cdot)+rv(s-h,\cdot)))(x)ds, 
$$
$$
v(t,x) = (S(t)v_0)(x) + \int_0^tS(t-s)(Mv(s,\cdot)+(b u(s,\cdot)- \epsilon v(s,\cdot))(1-v(s,\cdot)))(x)ds, 
$$
with $M= 2+r+br$ and the Gaussian-type action 
$$
(S(t)g)(x):= \frac{e^{-Mt}}{\sqrt{4\pi t}}\int_{-\infty}^{+\infty}e^{-(x-y)^2/(4t)}g(y)dy.
$$
We observe that $(0,0)$ and $(1,1)$ are  solutions of \eqref{2rr} and  the integral operators are monotone with respect to the functions with values in the interval $[0,1]$. Therefore the comparison principle can be applied to the mild solutions of  (\ref{2rr}) (actually, in the sense of definitions in \cite{FW,MS,SZ}, the reaction term in (\ref{2rr}) defines a quasi-monotone functional  with values in  $[0,1]^2$). 
Furthermore,  partial derivatives $u_x(t,x), v_x(t,x)$ are continuous and bounded in the variable $x\in \R$ for each  $t>0$; for instance, 
 $$|u_x(t,x)| \leq  \frac{e^{-Mt}}{\sqrt{\pi t}}|u_0|_\infty +  \int_0^t\frac{e^{-Ms}}{\sqrt{\pi s}}ds \sup_{x \in \R, s \in [0,t]} |u(s,x)(3+br-u(s,x)+rv(s-h,x))|.
$$
Thus  (\ref{2rr}) generates a continuous monotone semi-flow $\mathcal Q^t$ defined on the phase space $\mathcal C=  C([-h,0]\times \R, [0,1]^2)$ by 
$$
\mathcal Q^t(u_0,v_0)(s,x)= (u(t+s,x), v(t+s, x)), \quad s \in [-h,0], \ x \in \R.    
$$
Actually, arguing as in the proof of Theorem 1 in \cite{MS} (see also \cite[p. 517]{SZ}), we see that $(u(t,x), v(t,x))$ is a classical solution of (\ref{1hr}) for $t > h$. Therefore a mild monotone traveling wave $(u,v)(t,x)=(\phi, \psi)(x+ct)$  propagating with the speed $c$, i.e. 
$$
\mathcal Q^t(\phi, \psi)(s,x+cs)= (\phi, \psi)(x+c(t+s)), \quad s \in [-h,0], \ x \in \R, 
$$
is a classical traveling wave.

The system (\ref{2rr}) has 4 equilibria:
$$
O:= (0,0), \quad \beta:= (1,1); \quad \alpha_1:= (0,1); \quad  \alpha_2:= \left(\frac{\epsilon(r-1)}{br-\epsilon},\frac{b(r-1)}{br-\epsilon}\right). 
$$
We also consider the restriction  of the evolution system $\mathcal  Q^t$  on the subset of homogeneous initial data $\mathcal X_\beta$: if  $(\phi_0, \psi_0) \in \mathcal X_\beta$ and $(\phi(t),\psi(t)) = \mathcal  Q^t(\phi_0,\psi_0)(0,0)$, then 
 \begin{equation}\label{11hr}
\begin{array}{ll}
    \phi'(t) = \phi(t)(1-r-\phi(t)+r\psi(t-h)),
    &    \\
    \psi'(t) = (b\phi(t)- \epsilon \psi(t)) (1-\psi(t)).& 
\end{array}%
\end{equation}

The existence of the bistable wave 
will be obtained from  the next proposition  which essentially coincides with Theorems  3.4 and 5.1 in  \cite{FW} adapted to our particular situation: 
\begin{proposition}\label{FW} \cite{FW} Assume that for each $t >0$, the maps $\mathcal  Q^t$ and $\mathcal  Q = \mathcal  Q^1$ satisfy the assumptions 
\begin{itemize}
\item (A1) (Translation invariance) $T_a\mathcal  Q^t = \mathcal  Q^t T_a, \ a \in \R$, where $T_a: \mathcal C \to \mathcal C$ is a spatial shift operator: $T_a(u_0(\cdot, \cdot), v_0(\cdot, \cdot))= (u_0(\cdot, \cdot+a), v_0(\cdot, \cdot+a))$. 
\item  (A3) (Monotonicity) The semi-flow $\mathcal  Q^t$ is order preserving. 
\item  (A4') (Weak compactness)  There exists $s \in (0,h]$ such that:

(i) $\mathcal  Q[(u_0,v_0)](\theta,x)=(u_0,v_0)(\theta+s,x)$ whenever $\theta+s\leq 0$;

(ii) For any $\epsilon \in  (0, s)$, the set $\mathcal  Q[C_\beta]|_{[-s+\epsilon,0]\times \R}$ is precompact;

(iii) For any subset $J \subset C_\beta$ with $J (0,\cdot) \subset  C(\R,[0,1])$ being precompact, the set 
$\mathcal Q[J]|_{[-s,0]\times \R}$ is precompact.

\item  (A5) (Bistability) For each fixed $t>0$, the fixed points $O$ and $\beta$ are strongly stable from above and below,
respectively, for the map $\mathcal  Q^t:{\mathcal  X}_\beta \to {\mathcal X}_\beta$ and any two intermediate equilibria of $\mathcal  Q^t:{\mathcal  X}_\beta \to {\mathcal X}_\beta$ are unordered.

\item  (A6) (Counter-propagation) 
For each $j=1,2$, the time-one map $\mathcal  Q$ is such that 
$$
c_-^*(\alpha_j,\beta)+ c_+^*(O,\alpha_j)>0,
$$
where $c_+^*(O,\alpha_j)$ [$c_-^*(\alpha_j,\beta)$] is the rightward [leftward, respectively] asymptotic speed of propagation in the monostable subsystem $\mathcal  Q: [O, \alpha_j]_\mathcal C\to   [O, \alpha_j]_\mathcal C$ [$\mathcal  Q: [\alpha_j, \beta ]_\mathcal C\to  [\alpha_j, \beta ]_\mathcal C$, respectively]. 
\end{itemize}

Then there exists $c\in \R$ such that $\{ \mathcal  Q^t \}_{t \geq 0}$ admits a nondecreasing traveling wave with speed $c$ connecting $O$ to $\beta$.
\end{proposition}

The translation invariance and monotonicity property follow easily  from the integral equations, we note here that the coefficients of the system do not depend on $t$ and $x$. The proof of the weak compactness property is given in \cite{FW} (pages 2269-2270).
 
In the Appendix B, we analyse the dynamics of $\mathcal  Q^t$ on $\mathcal X_\beta$. In particular, for each $t >0$, we show that any two intermediate equilibria $\gamma_j \not \in \{O, \beta\}$, $j=1,2$,   of $\mathcal  Q^t$ are unordered.
Thus the last part of Assumption (A5) is satisfied and  we only need to show that the equilibria $O$ and $\beta$ meet the stability part of  Assumption (A5).  Since the linear parts of 
(\ref{11hr}) at  $O$ and $\beta$ lead to the reducible matrices (cf. \cite{Smith}), we cannot argue as in the proof of  \cite[Theorem 6.4]{FW}. In Claims 1 and 2 below we present our detailed proof of the required unilateral  stabilities for $O$ and $\beta$. 

Finally,  in view of the bidimensional character of system (\ref{2rr}), Assumption (A6) is the most complicated one. Note that 
the set of fixed points for all maps $\mathcal  Q^t: \mathcal X_\beta \to \mathcal X_\beta$, $t >0$ is precisely $\mathcal E: = \{O, \alpha_1, \alpha_2, \beta\}$. This explains the number of inequalities in (A6) according to  \cite{FW}.  The first inequality ($j=1$) is easy to check since both speeds $c_-^*(\alpha_1,\beta)$ and $c_+^*(O,\alpha_1)$ can be easily calculated,  see Claim 3. The second inequality ($j=2$), however, can not be verified following the techniques used in  \cite{FW}  for  the one-dimensional case.  Thus this case requires some additional ideas. In our paper, we prove the counter-propagation inequality for the interior fixed point $\alpha_2$ by  constructing in Appendices C and D  appropriate sub-solutions and super-solutions. This construction is used in Claim 4 below to establish the fulfilment of the Assumption (A6). 

\vspace{2mm}

\underline {Claim 1:}  The steady state $O$ is strongly stable from above. 

\vspace{2mm}
 
For $0<\epsilon < r-1$ and  small positive $m$, 
$$
0 < m  \leq m_0:= e^{-0.5\epsilon h}\min\left\{0.25, \left(1 -\frac{1}{r}- \frac{\epsilon}{2r}\right)\right\}, 
$$
and $\eta \in (0,1]$, we set 
$$\phi^*(t, \eta): =  \eta\phi^*(t):= \frac{\eta\epsilon m}{3b}   e^{-0.5\epsilon t}, \  \psi^*(t, \eta): =  \eta\psi^*(t):= \eta m e^{-0.5\epsilon t}.  $$
Then, for all $t \geq 0$,  $(\phi^*(t, \eta),\psi^*(t, \eta))$ is a strict super-solution in the sense that 
 \begin{equation}\label{difference1hra}
\begin{array}{ll}
    (\phi^*(t, \eta))' > \phi^*(t, \eta)(1-r-\phi^*(t, \eta)+r\psi^*(t-h, \eta)),
    &    \\
    (\psi^*(t, \eta))' > (b\phi^*(t, \eta)- \epsilon \psi^*(t, \eta)) (1-\psi^*(t, \eta)).& 
\end{array}%
\end{equation}

Consequently, the solution $(\phi(t, \eta),\psi(t,\eta))$ of  the initial value problem \\
$$(\phi(s, \eta),\psi(s, \eta))= (\phi^*(s, \eta),\psi^*(s, \eta)), \ s \in [-h,0],$$ 
for (\ref{11hr}) satisfies 
$$
(\phi(t, \eta),\psi(t, \eta)) < (\phi^*(t, \eta),\psi^*(t, \eta)), \quad t >0. 
$$
Indeed, this inequality obviously holds for all small $t >0$, and if $d>0$ is the leftmost point where the inequality fails
(for example, suppose that $\phi(d, \eta)= \phi^*(d, \eta)$, $\psi(s, \eta)\leq  \psi^*(s, \eta), \ s \leq d$,  then
$$
\phi(d, \eta)(1-r-\phi(d, \eta)+r\psi(d-h, \eta)) =\phi'(d, \eta)\geq 
$$
$$ (\phi^*)'(d, \eta) >  \phi(d, \eta)(1-r-\phi(d, \eta)+r\psi^*(d-h, \eta)), 
$$ 
a contradiction). 

Hence, for all $\sigma >0$, $s \in [-h,0]$, $\eta \in (0,1]$, it holds that
$$
\mathcal Q^\sigma(\eta\phi^*, \eta\psi^*)(s,0) \leq  (\eta\phi^*, \eta\psi^*)(\sigma +s,0) \ll \eta(\phi^*, \psi^*)(s). 
$$
Therefore, using terminology in \cite[p. 2247]{FW}, we can say that the steady state  $O$ is strongly stable from above. \hfill $\square$

\vspace{2mm}

\underline {Claim 2:}  The steady state $\beta$ is strongly stable from below. 

\vspace{2mm}

 Introducing the change of variables $\tilde \phi(t) = 1-\phi(t)$, 
$\tilde \psi(t) = 1-\psi(t)$, we transform (\ref{11hr}) (after suppressing the symbol "tilde" over variables) into the system 
 \begin{equation}\label{11hrl}
\begin{array}{ll}
 \psi'(t) = \psi(t)(-b +\epsilon -\epsilon \psi(t)+ b\phi(t)),
    &    \\   
     \phi'(t) = (r\psi(t-h)- \phi(t))(1- \phi(t)),& 
\end{array}%
\end{equation}
which essentially coincides with (\ref{11hr}). Thus, arguing as above, we find that the exponential functions  
 $$\phi^*(t, \eta): =  \eta\phi^*(t):= \eta mr^2   e^{-\delta t}, \  \psi^*(t, \eta): =  \eta\psi^*(t):= \eta m e^{-\delta t}  $$
 with $$\eta \in (0,1], \quad  0 < \delta < \min\left\{b-\epsilon, {{\frac{(r-1)\ln r}{r\ln r + (r-1)h}}}\right\},$$
 $$
 0 < m \leq m_1:=  \frac{1}{r^2} \min \left(1- \frac{\delta+\epsilon}{b}, 1 -\frac{r\delta}{r-e^{\delta h}}\right), 
 $$
are strict super-solutions for  (\ref{11hrl}).  Consequently, the solution $(\phi(t, \eta),\psi(t, \eta))$ of  the initial value problem $(\phi(s, \eta),\psi(s, \eta))= (\phi^*(s, \eta),\psi^*(s, \eta)), \ s \in [-h,0],$ for (\ref{11hrl}) satisfies 
$$
(\phi(t, \eta),\psi(t, \eta)) < (\phi^*(t, \eta),\psi^*(t, \eta)), \quad t >0. 
$$
Therefore
$$
(1-\phi(t, \eta), 1- \psi(t, \eta)) > (1-\phi^*(t, \eta),1-\psi^*(t, \eta)), \quad t >0, 
$$
so that for each $\sigma >0$, $s \in [-h,0]$, $\eta \in (0,1]$, it holds 
$$
\mathcal Q^\sigma(1-\eta\phi^*, 1-\eta\psi^*)(s,0) = (1-\phi(\sigma +s, \eta), 1- \psi(\sigma +s, \eta))\geq  $$
$$(1-\eta\phi^*, 1-\eta\psi^*)(\sigma +s) \gg (1,1)- \eta(\phi^*, \psi^*)(s). 
$$
By \cite{FW}, this means that the steady state  $\beta=(1,1)$ is strongly stable from below. $\square$

\vspace{2mm}

%%%%%%%%%%%%%%%%%%%%%%%%%%%%%%%%%%%%%%%%%%%%%%%%%%%%%%%%%%%%%%%%%%%%%%%%%%%%%%%%

\underline {Claim 3:} $ c_+^*(O,\alpha_1)=2\sqrt{\epsilon};\quad c_-^*(\alpha_1,\beta)=2.$. 

\vspace{2mm}

By \cite{FW}, to calculate $c_+^*(O,\alpha_1)$ it suffices to analyse the asymptotic behaviour of the solution $(u,v)$ to \eqref{2rr} with  smooth and nondecreasing in $x$ initial functions for which  
$$
u_0(s,x) \equiv 0, \ (s, x)\in [-h,0]\times \R, \quad  
$$  
$$
v_0(s,x) =  m_0e^{-0.5\epsilon s}, \ (s, x)\in [-h,0]\times (-\infty,0], 
$$
$$
v_0(s,x) =  1, \ (s, x)\in [-h,0]\times [1,+\infty),  
$$
where $m_0>0$ is defined in Claim 1.
Clearly, $u(t,x) \equiv 0$ for all $x \in \R$ and $t \geq 0$, and therefore the system (\ref{2rr})  simplifies to the scalar equation 
$$v_t(t,x) = v_{xx}(t,x)  -\epsilon v(t,x)(1-v(t,x)), \ v \geq 0, \ x \in \R.$$
Set $\tilde v(t,x)= 1-v(t,-x)$, then 
$$\tilde v_t(t,x) = \tilde v_{xx}(t,x)  +\epsilon \tilde v(t,x)(1-\tilde v(t,x)), \ \ x \in \R,$$
and the initial data are smooth, nondecreasing and satisfy 
$$
\tilde v(0,x) =  1- m_0,  \ x \geq 0, \quad 
\tilde v(0,x) =  0,  \  x \leq -1.  
$$
It is well known (e.g. see \cite[Example 2, p. 7]{bram}), that there exist  a function $c(t)$, $c(t) = 2\sqrt{\epsilon}\;t\;(1+o(1)), \ t \to +\infty,$ and a positive monotone function $\phi_1$ connecting $0$ and $1$  such that 
$$\sup_{x\in \R}|\tilde v(t,x) - \phi_1(x+c(t))| \to 0, \quad t \to +\infty. $$
Thus, for $x \leq ct,$ with $c < 2\sqrt{\epsilon}$, we find that 
$\sup_{x \leq ct}v(t,x)  \to 0, \quad t \to +\infty$
while  for $c >2\sqrt{\epsilon}$, we find that 
$v(t,ct)  \to 1, \quad t \to +\infty$. 
Since, by the definition in \cite{FW}, 
$$
c_+^*(O,\alpha_1):= \sup\{c \in \R:  \lim_{x \leq nc, n \to +\infty} v(n, x) =0 \}
$$
we obtain that $c_+^*(O,\alpha_1)= 2\sqrt{\epsilon}$. 

Next, we evaluate $c_-^*(\alpha_1,\beta)$. This time, we will study the asymptotic behaviour of the solution $(u,v)$ with initial functions  smooth and nondecreasing in $x$, for which
$$
v_0(s,x) \equiv 1, \ (s, x)\in [-h,0]\times \R, \quad  
$$  
$$
u_0(s,x) =  0, \ (s, x)\in [-h,0]\times (-\infty,-1], 
$$
$$
u_0(s,x) =  1-  m_1r^2   e^{-\delta s}, \ (s, x)\in [-h,0]\times [0,+\infty),
$$
where  $m_1 >0$  and $\delta >0$ are defined in Claim 2.
Thus $v(t,x) \equiv 1$ for all $x \in \R$ and $t \geq 0$, and the system (\ref{2rr})  simplifies to the  equation 
    $$u_t(t,x) =  u_{xx}(t,x)  + u(t,x)(1-u(t,x)),$$
    with nondecreasing and smooth  initial data satisfying $$
u_0(x) =  0, \ x \leq -1, \quad 
u_0(x) =  1-  m_1r^2,  \ x \geq 0. 
$$
In a consequence, arguing as in the previous case, we find that  
$$
c_-^*(\alpha_1,\beta):= \sup\{c \in \R:  \lim_{x \geq - nc, n \to +\infty} u(n, x) =1 \}=2. \hspace{4cm} \square
$$
%%%%%%%%%%%%%%%%%%%%%%%%%%%%%%%%%%%%%%%%%%%%%%%%%%%%%%%%%%%%%%%%%%%%%%%%%%%%%%%%%%%%%%%%%%%%%%%%%%%%%

\underline {Claim 4:} $c_-^*(\alpha_2,\beta)>0; \quad c_+^*(O,\alpha_2)\geq 0$. 

\vspace{2mm}

First, we prove the positivity of  $c_-^*(\alpha_2,\beta)$. Following \cite{FW}, we  study the asymptotic behaviour of solution $(u,v)$ of an initial value problem for (\ref{2rr}) with smooth and nondecreasing in $x$ initial functions which additionally satisfy  the equalities 
$$
u_0(s,x) \equiv \frac{\epsilon(r-1)}{br-\epsilon}=:\alpha_{21}, \ v_0(s,x) = \frac{b(r-1)}{br-\epsilon}=:\alpha_{22}, \ (s, x)\in [-h,0]\times (-\infty,-1], 
$$
$$
u_0(s,x) =  1-  mr^2   e^{-\delta s},\ v_0(s,x) =  1-  me^{-\delta s}, \ (s, x)\in [-h,0]\times [0,+\infty), 
$$
where $m \in (0,m_1)$ will be fixed later. 
In  Appendix C, for each sufficiently small positive $c$,  we indicate the $C^3$-smooth functions $U(t,x), V(t,x),$ which satisfy the relation $\beta=\lim_{n \to +\infty} (U,V)(n,-0.5nc)$ and the inequalities 
$$
U_t(t,x) <U_{xx}(t,x) + U(t,x)(1-r-U(t,x)+rV(t-h,x)), 
$$
$$
V_t(t,x)  < V_{xx}(t,x) +(bU(t,x)-\epsilon V(t,x))(1-V(t,x)), 
$$
in some open neighbourhood of the sector $\mathcal D = \{t \geq 0, x+ct \geq 0	\}$ of the upper half-plane $(x,t)$. In other words, using the terminology of \cite{FT}, the pair $(U(t,x), V(t,x))$ is a regular sub-solution of (\ref{2rr}) in the sector $\mathcal D$. 
In addition,  $U(t,x) > \alpha_{21}$ and  $V(t,x) > \alpha_{22}$ if $x+ct>0$, and $U(t,x) = \alpha_{21}$ and  $V(t,x) =\alpha_{22}$ if $x=-ct$.  Furthermore, for all sufficiently small $m>0$, it  holds that 
$$
U(s,x) \leq u_0(s,x), \quad V(s,x) \leq v_0(s,x),  \quad x \in \R, \ s \in [-h,0]. 
$$
Since $\alpha_{2}= (\alpha_{21},\alpha_{22})$ is the equilibrium of (\ref{2rr}), we  conclude
(again invoking the terminology in \cite{FT})
 that $$(\underline{U}(t,x), \underline{V}(t,x)) = \max\{(U(t,x), V(t,x)), \alpha_2\}$$
is a sub-solution of (\ref{2rr}) in the  upper half-plane $t \geq 0$ satisfying, 
for all sufficiently small $m>0$  the inequalities 
$$
\underline{U}(s,x) \leq u_0(s,x), \quad \underline{V}(s,x) \leq v_0(s,x),  \quad x \in \R, \ s \in [-h,0].
$$

We claim that the solution $(u,v)$ satisfies 
\begin{equation} \label{UVe}
\underline{U}(t,x) \leq u(t,x), \quad \underline{V}(t,x) \leq v(t,x),  \quad x \in \R, \ t \geq -h. 
\end{equation}
Indeed, let $\frak S$ be the set of all $S\geq 0$ such that the last inequalities hold for all $t \in [-h,S]$. Clearly, $0 \in \frak S$.  Set $T = \sup \frak S$ and suppose for a moment that $T\geq 0$ is a finite number. Then, for all $(t,x) \in W:=[T,T+h]\times \R$, 
$$
u_t(t,x) \geq u_{xx}(t,x) + u(t,x)(1-r-u(t,x)+r\underline{V}(t-h,x)), 
$$
$$
v_t(t,x)  = v_{xx}(t,x) +(bu(t,x)-\epsilon v(t,x))(1-v(t,x)), 
$$
so that  $(u(t,x),v(t,x))$ is a regular super-solution in $W$ for the system 
$$
u_t(t,x) =u_{xx}(t,x) + u(t,x)(1-r-u(t,x)+r\underline{V}(t-h,x)), 
$$
$$
v_t(t,x)  = v_{xx}(t,x) +(bu(t,x)-\epsilon v(t,x))(1-v(t,x)). 
$$
Since $(\underline{U}(t,x), \underline{V}(t,x))$ is an (irregular) sub-solution for the same system, Theorem 1 in \cite{FT} shows that actually 
$$
\underline{U}(t,x) < u(t,x), \quad \underline{V}(t,x) < v(t,x),  \quad x \in \R, \ t \in (T, T+h]. 
$$
This inequality contradicts the assumption of  finiteness of  $T$. Thus $T=+\infty$  and the estimates (\ref{UVe}) are proved. 
But then, for each sufficiently small  $c>0$, 
$$
\beta=\lim_{n \to +\infty} (U,V)(n,-0.5nc) =\lim_{x \geq - 0.5nc, n \to +\infty} (\underline{U},\underline{V})(n, x) 
\leq $$
$$\lim_{x \geq - 0.5nc, n \to +\infty} (u,v)(n, x)\leq \beta. 
$$
Thus $\lim_{x \geq - 0.5nc, n \to +\infty} (u,v)(n, x)=\beta$ and, consequently,  $c_-^*(\alpha_2,\beta)>0$.

Finally, to prove that   $c_+^*(O, \alpha_2) \geq 0$, we have to analyse the asymptotic behaviour of solution $(u,v)$ of (\ref{2rr})  with smooth and nondecreasing in $x$ initial functions,   which additionally satisfy
\begin{equation}\label{E27}
\begin{array}{r@{}l}
u_0(s,x) &{}=\alpha_{21}, \ v_0(s,x) =\alpha_{22}, \ (s, x)\in [-h,0]\times [0, +\infty), \\
u_0(s,x)  &{}  =   \frac{\epsilon m}{3b}   e^{-0.5\epsilon s},\ v_0(s,x) =  m e^{-0.5\epsilon s}, \ (s, x)\in [-h,0]\times (-\infty,-1], 
\end{array}
\end{equation}
where $m>0$ is sufficiently small. Following the same line of arguments as in the previous case, we will use a regular super-solution $(U(t,x), V(t,x))$ for (\ref{2rr}) constructed in Appendix D. Then we set 
$$(\overline{U}(t,x), \overline{V}(t,x)) = \min\{(U(t,x), V(t,x)), \alpha_2\},$$
choosing  $m$ to meet the inequalities 
$$
u_0(s,x)< \overline{U}(s,x) , v_0(s,x)) <\overline{V}(s,x), \quad (s,x) \in [-h,0] \times \R. 
$$
The construction of $(U(t,x), V(t,x))$  in Appendix D  assures that $(\overline{U}(t,x), \overline{V}(t,x)) =  \alpha_2$ for all $x \geq 0, t \geq -h$,  and $(\overline{U}(t,x), \overline{V}(t,x)) =  $ $(U(t,x), V(t,x))$ for all $x \leq 0, t \geq -h$. 

Then, arguing in the same way as below (\ref{UVe}), we obtain that the solution $(u,v)$ of the  initial value problem (\ref{E27})  for (\ref{2rr}) satisfies 
$$
(u(t,x), v(t,x)) \leq (\overline{U}(t,x), \overline{V}(t,x)),  \ t \geq -h, \ x \in \R. 
$$ 
Consequently, for each  $c<0$, we obtain  
$$
(0,0)=\lim_{n \to +\infty} (U,V)(n,0.5nc) =\lim_{x \leq 0.5nc, n \to +\infty} (\overline{U},\overline{V})(n, x) 
\geq $$
$$\lim_{x \leq 0.5nc, n \to +\infty} (u,v)(n, x)\geq (0,0). 
$$
Note here that the relation $\lim_{n \to +\infty} (U,V)(n,0.5nc)=(0,0)$ is assured by the form of  the super-solution $(U(t,x), V(t,x))$ constructed  in Appendix D. 
Thus $$\lim_{x \leq 0.5nc, n \to +\infty} (u,v)(n, x)=(0,0)$$ and, consequently,  
$$
c_+^*(O,\alpha_2):= \sup\{c \in \R:  \lim_{x \leq nc, n \to +\infty} (u,v)(n, x) =(0,0) \} \geq 0. 
$$
This completes the proof of Claim 4 and, respectively, of Theorem \ref{2.1}. 
\end{proof} 

%%%%%%%%%%%%%%%%%%%%%%%%%%%%%%%%%%%%%%%%%%%%%%%%%%%%%%%%%%%%%%%%%%%%%%%%%%%%%%%%%%%%%%%%%%%%%%%%%%%

\subsection{Limiting wavefront and the positivity of its speed}\label{s22}
As we  mentioned in the introduction, it is known from \cite{TPTa} that the propagation speed of the bistable wavefront to system (\ref{1}) is positive. This suggests the following:
\begin{lemma}\label{LE23}
The wavefront speeds $c_*(\epsilon)$ from Theorem \ref{2.1} satisfy
$$
+\infty > c_\star :=\liminf_{\epsilon \to 0^+}c_*(\epsilon) >0. 
$$
\end{lemma}
\begin{proof}
Fix  $\delta \in (0,1)$ and consider a sequence $\epsilon_j$ monotonically converging to $0,$ such that 
$c_\star :=\lim_{j \to +\infty}c_j,$  $c_j:=c_*(\epsilon_j)$. Let  $(\phi_j(x+c_jt), \psi_j(x+c_jt))$ is the sequence of the corresponding wavefronts. Without loss of generality, we can assume that $\phi_j(0)=\delta$ and  that 
the sequence $c_j$ is converging monotonically  to the number $c_\star$ (finite or infinite).

Clearly, the sequence $(\phi_j(t), \psi_j(t))$ is uniformly bounded on $\R$ and if $c_\star\not=0,$   
the sequence of  the derivatives $(\phi_j'(t), \psi_j'(t))$ is also uniformly bounded. Indeed, note that if $s$ is the critical point for $\phi_j'(t)$, then, for all large $j$, it holds that
$$
|\phi_j'(s)|= |c_j^{-1}\phi_j(s)(1-r-\phi_j(s)+r\psi_j(s-c_jh))| \leq  (1+|c_\star^{-1}|)(1+r). 
$$
This implies the boundedness of the sequence $\{\phi_j'\}$ in the sup-norm, the boundedness of $\{\psi_j'\}$ can be obtained in a similar way.
By differentiating (\ref{difference1hr}), we  also deduce the uniform boundedness of all higher
order derivatives of $\phi_j(t), \psi_j(t)$. Consequently, 
 $$(\phi_j(t), \psi_j(t),\phi_j'(t), \psi_j'(t), \phi_j''(t), \psi_j''(t)) \to (\phi_\star(t), \psi_\star(t),\phi_\star'(t), \psi_\star'(t), \phi_\star''(t), \psi_\star''(t)) $$
 uniformly on compact subsets of $\R$. In particular,  this proves that $c_\star$ is a finite number and $(\phi_\star(t), \psi_\star(t))$ is the solution of system  (\ref{b1hr}). Indeed, otherwise we will obtain from (\ref{difference1hr}) that $\psi'_\star(t)=0$,  $\phi'_\star(t)=0$  for all $t\in \R$, i.e.  $\phi_\star(t)\equiv \delta$, 
 $\psi_\star(t)\equiv 1$. However, this pair of functions does not satisfy  the first equation in  (\ref{b1hr}),  a contradiction.

Hence, $c_\star \in \R.$  Assume first that  $c_\star\not=0$. Then $(\phi_\star(t), \psi_\star(t))$  is a monotone traveling wave for (\ref{b1hr}) propagating with the speed $c_\star$. Since $\phi_\star(0)=\delta$,  {and the solutions are monotone, it follows from the equations (\ref{b1hr})}  that $(\phi_\star(+\infty), \psi_\star(+\infty))=(1,1)$ and $\phi_\star(-\infty) =0$. 
 If  $\psi_\star(-\infty)<1$, we  get from the second  equation in (\ref{b1hr}) that 
 \begin{equation}\label{arg}
c_\star(1-\psi_\star(-\infty))= b\int_\R\phi_\star(s)(1-\psi_\star(s))ds >0, 
\end{equation}
{which implies that $c_\star >0.$}
If $\psi_\star(-\infty)=1$, then  $\psi_\star(t)\equiv 1$ and the first equation in (\ref{b1hr}) implies that $c_\star \geq 2>0$. 

In the case $c_\star =0$, due to the Helly selection theorem, we can  assume that $(\phi_j(t), \psi_j(t))$ converges point-wise to the non-decreasing function $(\phi_\star(t), \psi_\star(t))$. Note that  each $(\phi_j(t), \psi_j(t))$ satisfies the integral equations 
$$
\phi_j(t) = \int_\R K_j(t-s)\phi_j(s)(-\phi_j(s)+r\psi_j(s-c_jh))ds,
$$
$$
\psi_j(t) = \int_\R K_j(t-s)((r-1)\psi_j(s)+ (b\phi_j(s)-\epsilon_j \psi_j(s))(1-\psi_j(s)))ds,
$$
where $K_j(s)= (c_j^2+4(r-1))^{-1/2} e^{\lambda_{\pm,j}s}$ if $\pm s \leq 0$ with
$$
\lambda_{\pm;j}=0.5(c_j\pm \sqrt{c_j^2+4(r-1)}).
$$
Applying the Lebesgue's dominated convergence theorem, we find that 
$$
\phi_\star(t) = \int_\R K_\star(t-s)\phi_\star(s)(-\phi_\star(s)+r\psi_\star(s))ds,
$$
$$
\psi_\star(t) = \int_\R K_\star(t-s)((r-1)\psi_\star(s)+ b\phi_\star(s) (1-\psi_\star(s)))ds,
$$
where $K_\star(s)= 0.5(r-1)^{-1/2} e^{\pm\sqrt{r-1}s}$ if $\pm s \leq 0$.  This
implies that $(\phi_\star(t), \psi_\star(t))$ is also a smooth solution of the system  (\ref{b1hr}) with $c_\star =0$. Repeating the arguments given around the formula (\ref{arg}), we conclude that  $c_\star >0$,  a contradiction.   
\end{proof} 
\begin{corollary}  \label{mdaF}  Suppose that $b >0,  \ r > 1$  and   $( \phi_\epsilon(t), \psi_\epsilon(t))$ is a solution to (\ref{difference1hr}).   Then, for all small $\epsilon >0$, 
\begin{equation}\label{oza}
0< \phi_\epsilon(t) < \psi_\epsilon(t) <1, \quad t \in \mathbb{R}.
\end{equation}
\end{corollary} 
\begin{proof} First, we note that, independently of the sign of $c_*(\epsilon)$,  $\phi_\epsilon(t)>0$ for all $t \in \R$. 
Indeed, the monotonicity of $\phi_\epsilon(t)\geq 0$ and the assumption that $\phi_\epsilon(s)= 0$ at some point $s$ implies $\phi_\epsilon(t)\equiv 0$ in virtue of the first equation of (\ref{difference1hr}), a contradiction. A similar argument shows that  $\psi_\epsilon(t) <1, \ t \in \mathbb{R},$ independently of the sign of $c_*(\epsilon)$. 

Set $z(t):=\phi_\epsilon(t)- \psi_\epsilon(t)$.  
We have  $z(-\infty)= z(+\infty)= 0,$
so that the non-negativity of $z$ at some points would imply the existence of 
such  a $\tau$ that $z(\tau)\geq  0,$  $z'(\tau)= 0,$ $z''(\tau)\leq 0$. But then $
\phi_\epsilon(\tau)\geq  \psi_\epsilon(\tau), 
$
$\psi_\epsilon(\tau-c_*(\epsilon) h) \leq \psi_\epsilon(\tau)$ (here we are using Lemma \ref{LE23})
and therefore 
\begin{eqnarray*}
0= z''(\tau) + (1-r-b)\phi_\epsilon(\tau) 
+\phi_\epsilon(\tau)(r\psi_\epsilon(\tau-c_*(\epsilon) h)-\phi_\epsilon(\tau))+  \\ b \phi_\epsilon(\tau)\psi_\epsilon(\tau) + \epsilon \psi_\epsilon(\tau)(1-\psi_\epsilon(\tau)), 
\end{eqnarray*}
$$
0 \leq 1-r-b
+r\psi_\epsilon(\tau-c_*(\epsilon) h)-\phi_\epsilon(\tau)+ b\psi_\epsilon(\tau) + 
\epsilon (1-\psi_\epsilon(\tau)) \leq 
$$
$$
1-r-b
+r\psi_\epsilon(\tau-c_*(\epsilon) h)-\psi_\epsilon(\tau)+ b\psi_\epsilon(\tau) + \epsilon (1-\psi_\epsilon(\tau)) \leq 
$$
$$
1-r-b + \epsilon + (r+b-1-\epsilon)\psi_\epsilon(\tau)= (r+b-1-\epsilon)(\psi_\epsilon(\tau)-1) <0, 
$$
a contradiction proving  (\ref{oza}). 
\end{proof}

\subsection{An upper bound for $c_\star$ and $c_*(\epsilon)$} \hfill

Modifying the arguments from \cite[Section 7]{Mur1}, we will establish here that  $c_\star$  and $c_*(\epsilon)$ are bounded from above with the constant $2$.  In particular, the inequality $c_\star <2$ implies that $\psi_\star(-\infty)<1$, see the discussion  around the formula (\ref{arg}). 
\begin{lemma}\label{LE25}{The wavefront speeds $c_\star$ and $c_*(\epsilon)$ (with small $\epsilon>0$)  belongs to the interval $(0, 2)$. }
\end{lemma}
\begin{proof} For a given $\epsilon>0,$ consider the speed $c_*(\epsilon)$, the situation with $c_\star$ is similar. On the contrary, suppose that $c_*(\epsilon)\geq 2$. 
Then \cite[Lemma 13]{TPTa} assures that the profile $\phi_\epsilon(t)$ has the following asymptotic representation at $-\infty$ (up to a translation): 
 $$\phi_\epsilon(t) =e^{\mu t} + O(e^{(2c_*(\epsilon)-\epsilon)t}), \quad \mbox{where} \ \mu = 0.5(c_*(\epsilon)+\sqrt{c_*(\epsilon)^2+4(r-1)}) >2.  $$
The solution $u=\phi_\epsilon(x+c_*(\epsilon)t)$ satisfies the relation 
$$u_{xx}(t,x)-  u_t(t,x)  + u(t,x)(1-u(t,x))=ru(t,x)z(t-h,x) > 0, \ \ z:=1-\psi_\epsilon(x+c_*(\epsilon)t),$$
and therefore, due to the maximum principle, $u(x,t)\leq \bar u(x,t)$, where $\bar u(x,t)$ is the solution of the initial 
value problem for 
\begin{equation}\label{KPP}
u_{xx}(t,x)-  u_t(t,x)  + u(t,x)(1-u(t,x))=0,
\end{equation}
with the initial data $\bar u(0,x)= \phi_\epsilon(x)$.

Since $\phi_\epsilon(t)e^{-2t} \to 0$ as $t \to -\infty$, due to the well known Bramson result (e.g. see \cite[Example 2, p. 7]{bram}), 
$$\lim_{t \to +\infty}\sup_{x \in \R}|\bar u(t,x)-\tilde \phi(x+2t-  (3/2)\ln t)|=0,$$ 
where $\tilde \phi$ denotes the profile of  the monotone minimal wavefront of the KPP-Fisher equation (\ref{KPP}). 
 Then  the  inequality $\phi_\epsilon(x+c_*(\epsilon)t) < \bar u(t,x)$ implies that 
$c_*(\epsilon) < 2$.
\end{proof}

%%%%%%%%%%%%%%%%%%%%%%%%%%%%%%%%%%%%%%%%%%%%%%%%%%%%%%%%%%%%%%%%%%%%%%%%%%%%%%%%%%%%%%%

\newpage 

\section{Proof of the main result}
\subsection{Perturbed system: { strongly unstable local manifold} at $(0,0)$} \hfill

Eigenvalues of  (\ref{difference1hr})  at $(0,0)$ are: $\lambda_1=\lambda_1(c, r)<0<\lambda_2=\lambda_2(c,r)$ (the first equation); 
$\mu_1=\mu_1(c, \varepsilon)\leq 0<\mu_2=\mu_2(c,\varepsilon)$ (the second equation); 
$$
\mu_{2,1}=\frac{c\pm\sqrt{c^2+4\varepsilon}}{2}, \quad \lambda_{2,1}=\frac{c\pm\sqrt{c^2+4(r-1)}}{2}.
$$
The system (\ref{difference1hr}) defines a semiflow $\Pi^t$ in the phase space of quads $$\frak{q}=(\phi_0, \phi'_0, \psi_0(\cdot), \psi'_0) \in \R^2\times C([-2h,0], \R)\times \R.$$
For simplicity, we will denote this space as $\R^3\times C$. Specifically, each such quad determines a unique solution 
$(\phi(t), \psi(t))$ of the system such that 
$$
\phi(0)=\phi_0, \ \phi'(0)=\phi'_0, \ \psi(s)=\psi_0(s), \ s \in [-2h,0], \ \psi'(0)=\psi'_0, 
$$
existing on some half-open maximal interval $I(\frak q)$. 
Then,  for $t \in I(\frak q)$, $$\Pi^t(\phi_0, \phi'_0, \psi_0(\cdot), \psi'_0) = (\phi(t), \phi'(t), \psi(t+\cdot), \psi'(t)).$$
Clearly, for all values of the parameters, the semiflow $\Pi^t$ has the fixed point zero. We will prove that the local unstable manifold $\gamma_\epsilon$ of this  zero equilibrium (existing for all $\epsilon >0$) depends continuously on $\epsilon$ and admits a continuous extension  up to the limit value $\epsilon =0$.

\begin{theorem} \label{invman} Suppose that $r>1, \epsilon, b , c>0$. There exists an open neighbourhood $\mathcal O \subset \R^2$ of $(0,0) \in \R^2$, $\epsilon_0 \in (0,1)$ and a continuous map
$$\gamma: [0, \epsilon_0] \times [0.75c_\star,2]\times \mathcal O  \to \R^3\times C,$$  such that, for each fixed $ (\epsilon,c) \in [0, \epsilon_0] \times [0.75c_\star,2]$, the map $\gamma$ is an embedding and the two-dimensional manifold  $\gamma(\epsilon, c, \mathcal O)$ is contained in the strongly unstable manifold  of the zero steady state (i.e. in the  manifold  of the initial data for the negative orbits  converging exponentially to $0$ as $t \to -\infty$, cf. \cite{HL}).  
\end{theorem}

\subsection{Proof of Theorem \ref{invman}} \label{Sub32}
In what follows, $C_w$ denotes the Banach space of continuous functions $(\phi(t), \psi(t))$ defined for all $t \leq 0$ and such that their norm below is finite: 
$$
\|(\phi, \psi)\|_w = \max\{|\phi|_w, |\psi|_w\}, \quad \mbox{where} \ \ |\phi|_w =\sup_{s \leq 0}|e^{-0.25c_\star s}\phi(s)| $$
Given $\rho, \epsilon \in (0,1)$ and $(\alpha,\beta)\in \R^2$ we will  also consider the closed ball $B_{\rho}\subset C_w$ of all elements 
satisfying the inequality 
$$
\|(\phi, \psi)\|_w \leq \rho. 
$$
We will look for solutions in  $B_{\rho}$ for the following system of integral equations 
$$
\phi(t)= \int_{-\infty}^te^{\lambda_1(t-s)}P(\phi,\psi)(s)ds +\int_{t}^0e^{\lambda_2(t-s)}P(\phi,\psi)(s)ds
$$
$$
+\left(\alpha- \int_{-\infty}^0e^{-\lambda_1s}P(\phi,\psi)(s)ds\right)e^{\lambda_2t}=: \mathcal{A}_1(\phi, \psi, \alpha)(t), 
$$
$$
\psi(t)= \int_{-\infty}^te^{\mu_1(t-s)}Q(\phi,\psi)(s)ds +\int_{t}^0e^{\mu_2(t-s)}Q(\phi,\psi)(s)ds
$$
$$
+\left(\beta- \int_{-\infty}^0e^{-\mu_1s}Q(\phi,\psi)(s)ds\right)e^{\mu_2t}=: \mathcal{A}_2(\phi, \psi, \alpha, \beta)(t), 
$$
where $\alpha, \beta$ are appropriately chosen parameters, $$
P(\phi,\psi)(s)= \frac{\phi(s)(-\phi(s)+r\psi(s-ch))}{\sqrt{c^2+4(r-1)}}, 
$$
$$
Q(\phi,\psi)(s)= \frac{b\mathcal{A}_1(\phi, \psi,\alpha)(s)(1-\psi(s))+ \epsilon \psi^2(s)}{\sqrt{c^2+4\epsilon}}.
$$

Note that $\phi(0)=\alpha$, $\psi(0)=\beta$, and each solution $(\phi(t),\psi(t))$ of the above integral equations satisfies the system (\ref{difference1hr}) for all $t \leq 0$.

Clearly, 
$$
|P(\phi,\psi)(s)| \leq \frac{1+r}{2\sqrt{r-1}} \rho^2 e^{0.5c_\star s}=:C_1\rho^2 e^{0.5c_\star s}, \quad s \leq 0, 
$$
$$
 \lambda_{1}=\frac{-2(r-1)}{c+\sqrt{c^2+4(r-1)}}\leq \frac{-(r-1)}{1+\sqrt{r}}=: \lambda_*^-, \  \lambda_{2}=\frac{c+\sqrt{c^2+4(r-1)}}{2} > c\geq 0.75c_\star,
$$
so that 
$$
|\mathcal{A}_1(\phi, \psi, \alpha)(t)| \leq  C_1\rho^2\left[\int_{-\infty}^te^{\lambda_1(t-s)}e^{0.5c_\star s}ds +\int_{t}^0e^{\lambda_2(t-s)}e^{0.5c_\star s}ds\right]+
$$
$$
\left(|\alpha|+ C_1\rho^2\int_{-\infty}^0e^{-\lambda_1s}e^{0.5c_\star s}ds\right)e^{\lambda_2t} \leq (|\alpha|+ C_2\rho^2)e^{0.5c_\star t}, \quad t \leq 0, 
$$
where 
$$
C_2 =\left(\frac{2}{0.5c_\star-\lambda_*^-} +\frac 4 c_\star\right)C_1
$$
does not depend on $\epsilon, \alpha, \beta$ and $c$. 
Clearly,  $|\mathcal{A}_1(\phi, \psi, \alpha)|_w \leq {\rho}$ once 
$$|\alpha|\leq \rho - C_2\rho^2.$$
Similarly, 
$$
|Q(\phi,\psi)(s)|\leq  \frac{2b(|\alpha|+ (C_2+b^{-1})\rho^2)e^{0.5c_\star s}}{c_\star}, \ \quad s \leq 0,
$$
$$
\mu_{1}=\frac{c-\sqrt{c^2+4\varepsilon}}{2} \leq 0, \quad \mu_{2}=\frac{c+\sqrt{c^2+4\varepsilon}}{2} \geq c \geq 0.75c_\star,$$
so that, for all  $t \leq 0$, 
$$
|\mathcal{A}_2(\phi, \psi, \alpha, \beta)(t)|  \leq \frac{16b(|\alpha|+ (C_2+b^{-1})\rho^2)e^{0.5c_\star t}}{c^2_\star} +|\beta|e^{0.5c_\star t}= $$
$$(|\alpha|C_3+|\beta|+ C_4\rho^2)e^{0.5c_\star t},  \quad \mbox{where} \ 
C_4 =16(C_2b+1)/c_\star^2, \ C_3= 16b/c_\star^2. 
$$
Consequently, $|\mathcal{A}_2(\phi, \psi, \alpha, \beta)|_w \leq {\rho}$ whenever 
$$
|\alpha|C_3+|\beta| \leq \rho -  C_4\rho^2. 
$$
In this way, for each $\rho \in (0, \min\{1,0.125/C_2,0.125/C_4\})$ there exists an open disc  $\mathcal O \subset \R^2$ centered  at $(0,0)$ (and not depending on the choice of $c \in [0.75c_\star,2] $ and  $\epsilon \in [0,1]$) such that the operator
$\mathcal{A}: = (\mathcal{A}_1, \mathcal{A}_2)$ maps $B_{\rho}$ into itself for each $(\alpha, \beta) \in \mathcal O$. 
Henceforth, we will fix one such $\rho >0$ and an associated disc $\mathcal O$. 
We claim that, under the above choice of $\rho$ and  $\mathcal O$, this mapping is a uniform contraction. Indeed, 
$$
e^{-0.5c_\star s}|P(\phi_2,\psi_2)(s) -P(\phi_1,\psi_1)(s)| \leq \rho \frac{(2+r)|\phi_2-\phi_1|_w +r|\psi_2-\psi_1|_w}{2\sqrt{r-1}} \leq 
$$
$$2C_1
\rho \|(\phi_2,\psi_2) - (\phi_1,\psi_1)\|_w, 
$$
$$
|\mathcal{A}_1(\phi_2, \psi_2,\alpha)- \mathcal{A}_1(\phi_1, \psi_1,\alpha)|_w\leq 2C_2
\rho \|(\phi_2,\psi_2) - (\phi_1,\psi_1)\|_w \leq $$ $$0.25\|(\phi_2,\psi_2) - (\phi_1,\psi_1)\|_w. 
$$
{Similarly, 
$$
|\mathcal{A}_2(\phi_2, \psi_2)- \mathcal{A}_2(\phi_1, \psi_1)|_w\leq 2C_4\rho(|\phi_2-\phi_1|_w +|\psi_2-\psi_1|_w) \leq 0.5\|(\phi_2,\psi_2) - (\phi_1,\psi_1)\|_w, $$}so that 
$$
\|\mathcal{A}(\phi_2, \psi_2,\alpha, \beta)- \mathcal{A}(\phi_1, \psi_1,\alpha, \beta)\|_w\leq 0.5\|(\phi_2,\psi_2) - (\phi_1,\psi_1)\|_w. 
$$
In this way, for each $(\alpha, \beta) \in \mathcal O$, we have proved the existence of a family of solutions $\phi(t,c,\epsilon, \alpha, \beta)$, 
$\psi(t,c,\epsilon, \alpha, \beta), t \leq 0,$ to system (\ref{difference1hr}) exponentially converging to $(0,0)$ as $t\to -\infty$ and such that 
$\phi(0,c,\epsilon, \alpha, \beta)=\alpha$, 
$\psi(0,c,\epsilon, \alpha, \beta)=\beta$. 

{Next, we observe that the uniform contraction $(\mathcal{A}_1, \mathcal{A}_2): B_{\rho}\to B_{\rho}$ depends continuously on the parameters $c,\epsilon, \alpha, \beta$ for each fixed $(\phi, \psi)\in B_\rho$. By using the Weierstrass M-test, this fact can be easily deduced from the uniform exponential estimates for the functions $P, Q$ in view of the point-wise continuity of all integrands in the expressions for operators $(\mathcal{A}_1, \mathcal{A}_2)$. For instance, let us show that 
$$
e^{-0.25c_\star t}\int_{-\infty}^te^{\mu_1(c, \epsilon)(t-s)}Q(\phi, \psi, \alpha, c, \epsilon)(s)ds \rightrightarrows $$
$$ e^{-0.25c_\star t}\int_{-\infty}^te^{\mu_1(c_0, \epsilon_0)(t-s)}Q(\phi, \psi,\alpha_0,c_0, \epsilon_0)(s)ds 
$$
uniformly in  $t \in (-\infty,0]$ once $(c, \epsilon, \alpha,\beta) \to (c_0, \epsilon_0, \alpha_0,\beta_0)$. Indeed, the integrand 
$${\mathcal I}(t,s,c, \epsilon, \alpha, \beta) = e^{-0.25c_\star t}e^{\mu_1(c, \epsilon)(t-s)}Q(\phi, \psi, \alpha, c, \epsilon)(s)$$ depends continuously on all variables and, for all $0\geq t \geq s$, $(\alpha, \beta) \in \mathcal O$, 
$$|{\mathcal I}(t,s,c, \epsilon, \alpha,\beta)|\leq |e^{-0.25c_\star t}Q(c, \epsilon)(s)|\leq  \frac{2b(|\alpha|+ (C_2+b^{-1})\rho^2)e^{0.5c_\star s}}{c_\star}e^{-0.25c_\star t}= $$
$$
\frac{2b(|\alpha|+ (C_2+b^{-1})\rho^2)e^{0.25c_\star s}}{c_\star}e^{-0.25c_\star (t-s)}\leq K e^{0.25c_\star s}, 
$$
where the constant $K$ does not depend on $t,s,c, \epsilon, \alpha,\beta$. 
}

This implies that the mapping 
$$
(\phi, \psi): [0,0.5] \times [0.75c_\star,2]\times \mathcal O \to B_\rho
$$
is continuous, 
cf. \cite[Section 1.2.6]{H}. 
Particularly, for each fixed pair $(\epsilon, c) \in [0,0.5] \times [0.75c_\star,2]$,
we can define the continuous injection
$$\gamma: [0,0.5] \times [0.75c_\star,2]\times \mathcal O  \to \R^3\times C,$$
by 
$$
\gamma(\epsilon, c, \alpha, \beta)= (\alpha, \phi'(0,c,\epsilon, \alpha, \beta), \psi(\cdot ,c,\epsilon, \alpha, \beta), \psi'(0,c,\epsilon, \alpha, \beta)). 
$$
Thus the set  $\gamma(\epsilon, c, \mathcal O )$ is 
a two-dimensional manifold, continuously depending on the parameters $\epsilon, c$. It is a part of a bigger two-dimensional negatively invariant manifold defined as 
$$
\Gamma(\epsilon, c)= \{(\phi(s,c,\epsilon, \alpha, \beta), \phi'(s,c,\epsilon, \alpha, \beta), \psi(s+\cdot ,c,\epsilon, \alpha, \beta), \psi'(s,c,\epsilon, \alpha, \beta)), s \leq 0\}.  
$$
This completes the proof of Theorem  \ref{invman}. \hfill $\Box$
\begin{remark}Fix some $\tau$ and consider the functions $$\phi_\tau(t,c,\epsilon, \alpha, \beta) = \phi(t+\tau,c,\epsilon, \alpha, \beta), \quad 
\psi_\tau(t,c,\epsilon, \alpha, \beta) = \psi(t+\tau,c,\epsilon, \alpha, \beta), \quad t \leq 0.$$
It is clear that  $$
\|(\phi_\tau, \psi_\tau)\|_w \leq \rho e^{0.25c_\star \tau} \leq \rho, \  \mbox{if} \ \tau \leq 0. 
$$
A straightforward verification also shows that 
$$
\mathcal{A}_1(\phi_\tau, \psi_\tau, \phi(\tau))(t) = \phi_\tau(t), \quad \mathcal{A}_2(\phi_\tau, \psi_\tau, \phi(\tau), \psi(\tau))(t) = \psi_\tau(t). 
$$
Hence, if $(\phi(\tau), \psi(\tau)) \in \mathcal O$, then the uniqueness statement of the  Banach contraction principle  implies that 
\begin{eqnarray}\label{pp}
\phi(t+\tau,c,\epsilon, \phi(0), \psi(0)) = \phi(t,c,\epsilon, \phi(\tau), \psi(\tau)), \\
\psi(t+\tau,c,\epsilon,  \phi(0), \psi(0)) = \psi(t,c,\epsilon,  \phi(\tau), \psi(\tau)). \nonumber
\end{eqnarray}
Note that the latter equality holds not only for $t \leq 0$ but also for all positive $t$, for which these functions are defined. 
\end{remark}
%%%%%%%%%%%%%%%%%%%%%%%%%%%%%%%%%%%%%%%%%%%%%%%%%%%%%%%%%%%%%%%%%%%%%%%%%%%%%%%%%%%%%%%%%%%%%%%%%%%%

\subsection{Proof of Theorem \ref{mainT}} We will take $\rho >0$ and an associated disc $\mathcal O$ fixed in Subsection \ref{Sub32}.  Then, for some fixed  $\delta \in (0,0.5)$, we consider the sequence  $(\phi_j(t), \psi_j(t))$ of bistable wavefronts defined in the proof of Lemma \ref{LE23} (Subsection \ref{s22}).  This sequence converges uniformly on $\R$  to some monotone traveling wave  $(\phi_\star(t), \psi_\star(t))$  propagating with the speed $c_\star$. We have already established that $(\phi_\star(+\infty), \psi_\star(+\infty))=(1,1)$ and $\phi_\star(-\infty) =0$,  $\psi_\star(-\infty)\in [0,1)$. Our goal is to prove that $\psi_\star(-\infty) =0$. 

To proceed, we  take $\delta_1\in (0,\delta)$ such that the Euclidean $\delta_1$-neighbourhood $U$ of $(0,0)$ is contained in $\mathcal O$.  
Then for every $(\phi_j(t), \psi_j(t)),$   $j=1,2,...$ we can find the unique  $s_j < 0$ such that $(\alpha_j, \beta_j):=(\phi_j(s_j), \psi_j(s_j)) \in \partial U$. Clearly, we can assume that $(\alpha_j, \beta_j) \to (\alpha_\star, \beta_\star)$ and that $s_j$ converges to some extended real number $s_*\leq 0$.  Next, in view of the asymptotic behaviour of $(\phi_j(t), \psi_j(t))$ at $-\infty$ (cf. the proof of Lemma \ref{LE25}), there exists  sufficiently large negative $\tau_j$ such 
that $(\hat\phi_j(\cdot), \hat\psi_j(\cdot)) := (\phi_j(\cdot+s_j+\tau_j), \psi_j(\cdot+s_j+\tau_j)) \in B_\rho$  and $(\tilde\alpha_j, \tilde\beta_j): = (\phi_j(s_j+\tau_j), \psi_j(s_j+\tau_j)) \in U \subset \mathcal O$. Since  $(\hat\phi_j(t), \hat\psi_j(t))$ solves the integral equations 
$$
\mathcal{A}_1(\hat\phi_j, \hat\psi_j, \tilde\alpha_j)(t) = \hat\phi_j(t), \quad \mathcal{A}_2(\hat\phi_j, \hat\psi_j, \tilde\alpha_j, \tilde\beta_j)(t) = \hat\psi_j(t),$$
the uniqueness part of the Banach contraction principle yields 
$$
(\phi_j(t+s_j+\tau_j), \psi_j(t+s_j+\tau_j)) = (\phi(t,c_j,\epsilon_j, \tilde \alpha_j, \tilde \beta_j), \psi(t,c_j,\epsilon_j, \tilde \alpha_j, \tilde \beta_j)). 
$$
Now, observe that, in view of the monotonicity of the wavefronts,  $$
(\phi_j(t+s_j+\tau_j), \psi_j(t+s_j+\tau_j)) \in U, \quad  t \in [0, -\tau_j].$$
Thus the functions  $\phi(t,c_j,\epsilon_j, \tilde \alpha_j, \tilde \beta_j), \psi(t,c_j,\epsilon_j, \tilde \alpha_j, \tilde \beta_j)$ (extended for all $t\in \R$ as solutions of (\ref{difference1hr})) satisfy 
$$
 (\phi(t,c_j,\epsilon_j, \tilde \alpha_j, \tilde \beta_j), \psi(t,c_j,\epsilon_j, \tilde \alpha_j, \tilde \beta_j)) \in U, \quad t \in [0, -\tau_j], 
$$
and 
$$
 (\phi(-\tau_j,c_j,\epsilon_j, \tilde \alpha_j, \tilde \beta_j), \psi(-\tau_j,c_j,\epsilon_j, \tilde \alpha_j, \tilde \beta_j)) = (\alpha_j, \beta_j). 
$$

Since (\ref{pp}) gives 
$$
 (\phi(t,c_j,\epsilon_j, \tilde \alpha_j, \tilde \beta_j), \psi(t,c_j,\epsilon_j, \tilde \alpha_j, \tilde \beta_j))=  (\phi(t+\tau_j,c_j,\epsilon_j,  \alpha_j, \beta_j), \psi(t+\tau_j,c_j,\epsilon_j, \alpha_j, \beta_j)), 
$$
we obtain 
$$
(\phi_j(t+s_j), \psi_j(t+s_j)) = (\phi(t,c_j,\epsilon_j,  \alpha_j, \beta_j), \psi(t,c_j,\epsilon_j,  \alpha_j,  \beta_j)). 
$$
Therefore we can conclude that, uniformly on $(-\infty, 0]$,  %there exists the following limit 
$$
\lim_{j \to +\infty}(\phi_j(t+s_j), \psi_j(t+s_j)) = (\phi(t,c_\star,0, \alpha_\star, \beta_\star), \psi(t,c_\star,0, \alpha_\star, \beta_\star))
$$
with  monotone  $(\phi(t,c_\star,0, \alpha_\star, \beta_\star), \psi(t,c_\star,0, \alpha_\star, \beta_\star)), \ t \leq 0,$ exponentially converging to $(0,0)$ as $t \to -\infty$ and with  $(\alpha_\star, \beta_\star) \in \partial U$. This shows that 
$s_*$ is a finite number (otherwise $\phi(t,c_\star,0, \alpha_\star, \beta_\star), \psi(t,c_\star,0, \alpha_\star, \beta_\star)$ are constant functions and $(\alpha_\star, \beta_\star)$ coincides with $(0,0)\not\in \partial U$) and that the function
 $$(\phi_\star(t+s_*), \psi_\star(t+s_*)) =(\phi(t,c_\star,0, \alpha_\star, \beta_\star), \psi(t,c_\star,0, \alpha_\star, \beta_\star))
$$  is a bistable wave connecting $(0,0)$ and $(1,1)$. \hfill $\square$

\appendix
\section{Bistable character of equation (\ref{1r}) for $r>1$}\label{ApA} \hfill 

In this section, we prove that the zero equilibrium of system (\ref{1hr}) is stable. In fact, we will show that each solution 
$(\phi(t), \psi(t))$ of the system starting in some small box $[0,\delta]^2$ behaves in the following way: the component 
$\phi(t)$ exponentially decreases to $0$ and $\psi(t)$ is strictly increasing but remains in the interval $[0, M\delta]$, where $M =2(1 + b/(r-1))$. 
Indeed, suppose that $\delta \in (0, 1-1/r)$ is small enough to satisfy 
$$
1- (1-\delta)e^{-b\delta/(r-1-r\delta)} < M\delta. 
$$
Consider the non-negative initial conditions $\phi(0) \in [0,\delta]$, $\psi(s) \in [0,\delta]$, $s \in [-h,0]$. 
If $\phi(0)=0$ then $\phi(t) = 0$ for all $t \geq 0$ and therefore $\psi(t) = \psi(0)$, $t \geq 0$.  Hence, the  stability claim holds trivially. 

Assume that 
$\phi(0)>0$. As a consequence, $\psi(t)$ will increase on $\R_+$ and $\phi(t), \psi(t)$ will take values in $[0, M \delta)$ for all $t$ from  some maximal non-empty interval $[0, T_\delta)$.   This implies that $\phi'(t) < \phi(t)(1-r+r\delta)$ so that $\phi(t)$ is exponentially decreasing on  $[0, T_\delta)$: $\phi(t) < \phi(0)\exp((1-r+r\delta)t)$,  $t \in (0, T_\delta)$. Thus in the case when $T_\delta$ is finite, it is determined by the condition $\psi(T_\delta)=M\delta$.  However, by  integrating the second equation in (\ref{1hr}), we find that 
$$
M\delta  =\psi(T_\delta)= 1- (1-\psi(0))\exp (-b\int_0^{T_\delta}\phi(s)ds)  < $$
$$ 1- (1-\delta)\exp (-b\delta\int_0^{+\infty}e^{(1-r+r\delta)s}ds)=  1- (1-\delta)e^{-b\delta/(r-1-r\delta)} < M \delta, 
$$
a contradiction. Thus $T_\delta=+\infty$ and we obtain that $\phi(t)$ is exponentially decreasing to $0$ and $\psi(t)$ is strictly increasing and bounded by $M\delta$. This proves our  initial claim.  
\begin{figure}[tbhp]
\centering
\subfloat[$b=2$, $r=0.5,$ monostable]{\label{fig:1a}\includegraphics[width=4.3cm]{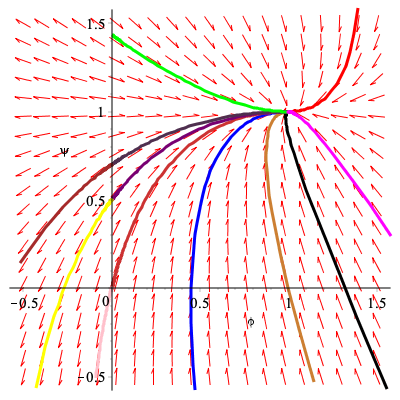}}
\subfloat[$b=2$, $r=2,$ bistable cases]{\label{fig:1b}\includegraphics[width=4.3cm]{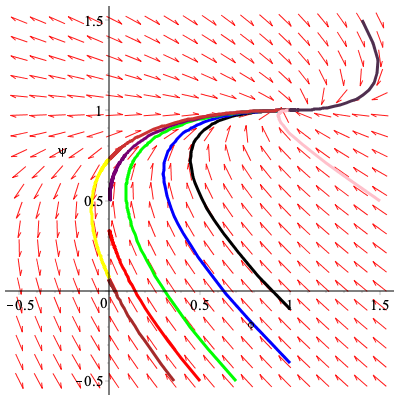}}
\caption{The phase plane analysis for \eqref{1hr} with $h=0$.}
\label{Fig1}
\vspace{-3mm}
\end{figure}

Figure 1(b) illustrates graphically the above discussion in the non-delayed case ($h=0$). 
 Note that the lines $\phi=0$ (vertical) and $\psi=1$ (horizontal) on Figure 1 are invariant. The line $\phi=0$ consists of steady states $(0, \zeta).$  After appropriate transformations, 
  each $\zeta \leq 1$ corresponds to the level $r(\zeta)=(1-\zeta)r$ of the bromide ion in the original system. One can ask about the existence of traveling waves connecting equilibria  $(0, \zeta)$ and $(1,1)$ in (\ref{1r}).   The homogeneous steady state  $(1,1)$ is always asymptotically stable and hyperbolic,  while the steady state 
$(0,0)$ is unstable  and nonhyberbolic  for $r<1$ and becomes nonhyperbolic stable for $r>1.$ 
Furthermore, for $r<1$ all equilibria $(0,\zeta)$,  $\zeta \in (0, 1)$,  
are unstable. 
For $r>1$ the equilibria $(0,\zeta)$ for $\zeta \in \left(-\infty, 1-1/r \right) $ are stable and the equilibria $(0,\zeta) $ for $\zeta \in \left(1-1/r, \infty \right) $ are unstable.   For the special value  $\zeta_0 = 1-1/r$ we obtain $r(\zeta_0)=1.$ The heteroclinic connections between the fixed points $(0, \zeta)$ and $(1,1)$ with $\zeta\in [\zeta_0,1)$ are  monostable, and the connections between the points  $(0,\zeta)$ and $(1,1)$ with $\zeta < \zeta_0$ are bistable (cf. \cite[p. 333]{Volp}).  Hence, 
the phase space analysis shows  that with respect to the equilibria $(0,0)$ and $(1,1)$,  system  (\ref{1r}) changes its character from a monostable system to a bistable one  at $r=1$ (see Figure 1).

\begin{figure}[tbhp]
\centering
\subfloat{\label{fig:3a}\includegraphics[width=4.3cm]{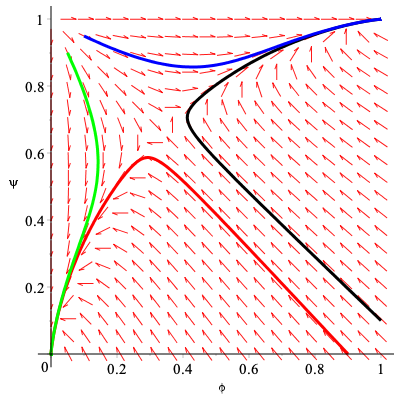}}
\subfloat{\label{fig:3b}\includegraphics[width=4.3cm]{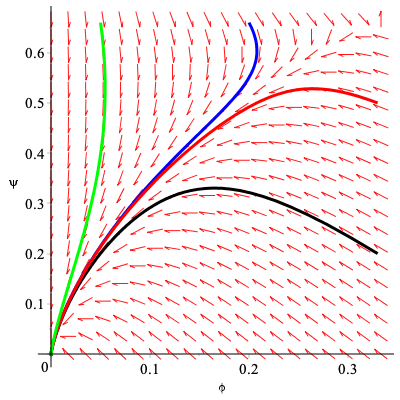}}
\subfloat{\label{fig:3c}\includegraphics[width=4.3cm]{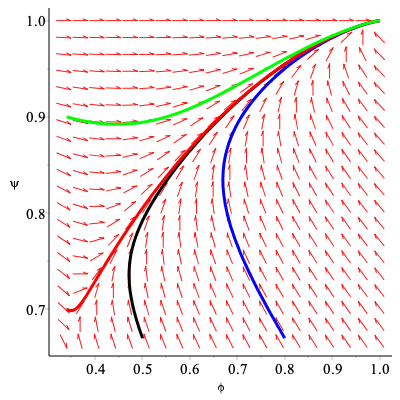}}
\caption{The phase plane analysis for \eqref{2hr} with $b=2$, $r=2$, $\epsilon = 1$ and $h=0$: (a) the whole system,   (b) and (c) two monostable subsystems.}
\label{Fig3}
\end{figure}
The perturbed system \eqref{2hr}
 \begin{equation}\label{2hr}
\begin{array}{ll}
    \phi'(t) = \phi(t)(1-r-\phi(t)+r\psi(t-h)),
    &    \\
    \psi'(t) = (b\phi(t)-\epsilon \psi(t))(1-\psi(t)), \ (\phi, \psi) \in \R^2,& 
\end{array}%
\end{equation}
has  for $h=0$ and $\epsilon>0$ two  stable hyperbolic equilibria $O$ and $\beta$ and two unstable equilibria $\alpha_1$ and $\alpha_2$ of a saddle-node type, see Figure 2 (a). 
Consequently, system \eqref{2hr} is of the classical bistable type and easier to analyse.
Note also that  if we restrict the picture to the invariant regions $\left[0,\frac{\epsilon(r-1)}{rb-\epsilon}\right] \times \left[ 0, \frac{b(r-1)}{rb-\epsilon} \right]$ and $\left[\frac{\epsilon(r-1)}{rb-\epsilon},1\right] \times \left[  \frac{b(r-1)}{rb-\epsilon} , 1\right]$ system \eqref{2hr} splits into two monostable subsystems, see Figure 2 (b) and (c).

%%%%%%%%%%%%%%%%%%%%%%%%%%%%%%%%%%%%%%%%%%%%%%%%%%%%%%%%%%%%%%%%%%%%%%%%%%%%%%%%%%%%%%%%%%%%%%%%%%%%

\section{Analysis of  system  (\ref{2hr}) with $r>1, \ \epsilon >0$}\label{ApB} \hfill

Recall that system (\ref{2hr})  has exactly four equilibria $$O=(0,0),\ \beta=(1,1), \ \alpha_1=(0,1),\  \alpha_2=\left(\frac{\epsilon(r-1)}{rb-\epsilon},\frac{b(r-1)}{rb-\epsilon}\right).$$  
The right-hand side of (\ref{2hr}) defines a continuous functional $F(\phi, \psi)$ defined on the space of continuous functions $C([-h,0], \R^2)$.  The restriction $$F=(F_1,F_2):  {\mathcal X}_\beta = C([-h,0], [0,1]^2) \to \R^2$$ of this functional is quasi-monotone \cite[p. 78]{Smith} in the sense that 
$$
F_1(\phi_1, \psi_1) \leq F_1(\phi_2, \psi_2), \quad \mbox{once} \ \phi_1(0) = \phi_2(0), \  \phi_1(s) \leq \phi_2(s), \  \psi_1(s)\leq \psi_2(s).   
$$
$$
F_2(\phi_1, \psi_1) \leq F_2(\phi_2, \psi_2), \quad \mbox{once} \ \psi_1(0) = \psi_2(0),\  \phi_1(s) \leq \phi_2(s), \  \psi_1(s)\leq \psi_2(s).   
$$
By \cite[Theorem 1.1, p. 78]{Smith}, system (\ref{2hr})  generates a monotone semi-flow on the phase space ${\mathcal X}_\beta$. In particular, the order intervals $[O, \alpha_2]_{{\mathcal X}_\beta}$ and  $[\alpha_2, \beta]_{{\mathcal X}_\beta}$ are invariant. 

Now, let $(\phi_t,\psi_t), \ t \geq 0,$ be the solution of (\ref{2hr})  with the initial data $(\phi_0,\psi_0) \in  {\mathcal X}_\beta$. Suppose that $\phi_p(0)=0$ for some $p \geq 0$. Then from the first 
equation of  (\ref{2hr})  we find that $\phi_t(0) =0$ for all $t \geq p$. Using the second equation of the system we find that either $\psi_t(0) \to 0$ as $t \to +\infty$ or $\psi_t(0) =1, \ t \geq p$. 

Similarly, if $\psi_q(0)=1$ for some $q \geq 0$, then from the second 
equation of  (\ref{2hr})  we find that $\psi_t(0) =1$ for all $t \geq q$. Then using the first equation of the system we find that either $\phi_t(0) \to 1$ as $t \to +\infty$ or $\phi_t(0) =0, \ t \geq q+h$. 

Next, if $\phi_p(0)=1$ for some $p > h$, we may conclude that $\phi_p'(0)=0$ and therefore $\psi_{p-h}(0)=1$. By the previous paragraph, $(\phi_t,\psi_t)=\beta$ for all $t \geq p$. 

Finally,  if $\psi_q(0)=0$ for some $q \geq 0$, then $\psi'_q(0)=0$ and therefore,  by the second equation in (\ref{2hr}), 
 $b\phi_q(0)= \epsilon \psi_q(0)=0$. Thus we may conclude that  $(\phi_t,\psi_t)=O$ for all $t \geq q$. 
  
 The above analysis shows that each non-convergent solution $(\phi_t,\psi_t), \ t \geq 0$, of  (\ref{2hr})  satisfies, for some $T>0$, the inequalities 
 $$
 0 < \phi_t(0), \psi_t(0) < 1, \quad t \geq T. 
 $$ 
Let us fix some $\tau>0$ and suppose that, for some $(\phi_0,\psi_0) \in {\mathcal X}_\beta \setminus \{O, \alpha_1, \alpha_2, \beta\}$ it holds  
$$
(\phi_\tau,\psi_\tau) = (\phi_0,\psi_0). 
$$
Thus we obtain  a $\tau$-periodic solution $(\phi(t),\psi(t)) = (\phi_t(0),\psi_t(0))$ with values in the open interval $(0,1)$. Clearly, this implies that 
$$
\int_0^\tau(1-r-\phi(s)+r\psi(s-h))ds = \int_0^\tau(b\phi(s)- \epsilon \psi(s))ds= 0, 
$$
from where the average values $(\bar\phi,\bar \psi)$ of $(\phi(t),\psi(t))$ satisfy 
$$
(\bar\phi,\bar \psi)= \alpha_2. 
$$
Consequently,  the elements $(\phi_0,\psi_0)$ and $\alpha_2$ are unordered in view of the monotonicity of the system. 
Clearly,   the pair $(\phi_0,\psi_0)$ and $\alpha_1$ also cannot be ordered. 

By the same reason, if for another initial data $(\phi_0^*,\psi_0^*) \in  {\mathcal X}_\beta \setminus \{O, \alpha_1, \alpha_2, \beta\}$ it holds  
$$
(\phi^*_\tau,\psi^*_\tau) = (\phi^*_0,\psi^*_0), 
$$
we obtain that $$
(\bar\phi^*,\bar \psi^*)= \alpha_2= (\bar\phi,\bar \psi). 
$$
If the elements $(\phi_0^*,\psi_0^*)$ and $(\phi_0,\psi_0)$ were ordered, (e.g. $(\phi_0^*,\psi_0^*) \geq (\phi_0,\psi_0)$), 
in view of the monotonicity of the system, the inequalities 
$$
\phi^*(t) \geq \phi(t), \psi^*(t) \geq \psi(t), \quad t \geq 0, 
$$ 
should be satisfied. This, however, would lead to a contradiction $(\bar\phi^*,\bar \psi^*)> (\bar\phi,\bar \psi)$ since the corresponding periodic solutions are different.

%%%%%%%%%%%%%%%%%%%%%%%%%%%%%%%%%%%%%%%%%%%%%%%%%%%%%%%%%%%%%%%%%%%%%%%%%%%%%%%%%%%%%%%%%%%%%%%%%%%

\section{Construction of a sub-solution} \hfill

Assume that $\epsilon \in (0, rb/7)$ is small enough to imply $\alpha_{21} < 1/6$ and consider the following function $$
U(x)=  \frac{1}{1+\rho_1e^{-\lambda x}}$$ with 
$$\rho_1= \frac{r(b-\epsilon)}{\epsilon(r-1)}$$
and $\lambda $ sufficiently small, which will be specified later. 

Note that $
U(0)= \alpha_{21}$ and $U(x)$ satisfies the differential equations
\begin{equation}\label{log}
u'(x)= \lambda u(x)(1-u(x)), \quad u''(x)= \lambda^2 u(x)(1-u(x))(1-2u(x)). 
\end{equation}

 Since $1-3U(0) >0$, for all small $\epsilon, \lambda >0$ there exists a unique point 
$$x_*= \frac{1}{\lambda}\ln \frac{r(b-\epsilon)}{2\epsilon(r-1)} \gg 1$$
such that $1-3U(x_*)=0$. Clearly, $1-3U(x) >0$ for all $x \in [0, x_*)$.    

Consider also $C^3$-smooth nonnegative function $\gamma(x)= \gamma(x, \epsilon, \lambda)$ defined as  
$
\gamma(x) = 0 \ \mbox{if} \ x \leq x_*,$ $\gamma(x) = (1-3U(x))^4(1-U(x)) \ \mbox{if} \ x \geq x_*$. It is easy to find that  $$\gamma''(x)= \lambda^2U(x)(1-U(x))P(U(x)), \quad  \gamma(x) \sim 16 \rho_1e^{-\lambda x}, \ x \to +\infty, 
$$
where $P$ is some real polynomial of degree 5.  
Let $\theta \in (0, 1/16)$ be a small parameter and $V(x)$ be defined from  the equation  
$$
1-r-U(x)+rV(x) =\theta \gamma (x). 
$$
Then 
$V(0) =\alpha_{22}
$, $1-V(x) \sim \rho_1r^{-1}e^{-\lambda x}(1-16\theta), \ x \to +\infty, $ $V'(x)= r^{-1}U'(x)>0$, $x \leq x_*$,  and, for some polynomial $Q$ of degree 4, it holds that 
$$
V'(x)= r^{-1}(U'(x)+ \theta \gamma'(x))= r^{-1}U'(x)(1+ \theta Q(U(x))) >0, \quad x \geq x_*, 
$$
$$
V'(x)\leq  U'(x), \quad x \geq 0, \quad 1+ \theta Q(U(x)) \in [0.5,1.5], 
$$
for all small $\theta >0$. Therefore 
$$
0 \leq \theta \gamma (x) = 1-U(x)- r(1-V(x)) \leq 1, \quad x \geq 0. 
$$
Next, for all $x \geq 0$ and $\lambda \in (0, \sqrt{\theta}/4)$, by using (\ref{log}), we find that 
$$
U''(x)+U(x)(1-r-U(x)+rV(x))= 
$$
$$
= U(x)\left(\lambda^2(1-2U(x))(1-U(x)) + \theta \gamma (x)\right) >0. 
$$
Note here that  $\Pi(x):=\lambda^2(1-2U(x))(1-U(x)) + \theta \gamma (x) = \lambda^2(1-2U(x))(1-U(x)) \geq 2\lambda^2/9$ for $x \in [0, x_*]$, $\Pi(x)= (1-U(x))(\lambda^2(1-2U(x))+ \theta  (1-3U(x))^4) >0$ when $U(x) \in [1/3,1/2]$, and $\Pi(x)= (1-U(x))(\lambda^2(1-2U(x))+ \theta  (1-3U(x))^4) \geq $ \mbox{ $(1-U(x))(\theta/16 - \lambda^2)>0$} when $U(x) \in [1/2,1)$. 

Finally, for $x \in [0, x_*]$, 
$$
V''(x) +(bU(x)-\epsilon V(x))(1-V(x))= \lambda^2r^{-1}U(x)(1-2U(x))(1-U(x)) + 
$$
$$
r^{-1}(rb- \epsilon)\left(U(x)-\alpha_{21} \right)(1-V(x)) >0, 
$$
as well as, for $x \geq x_*$, and all sufficiently small $\lambda >0$, 
$$
V''(x) +(bU(x)-\epsilon V(x))(1-V(x))\geq \lambda^2r^{-1}U(x)(1-2U(x))(1-U(x)) + \theta r^{-1}\gamma''(x)+ 
$$
$$
r^{-1}(rb- \epsilon)\left(\frac 1 3 -\alpha_{21}- \epsilon\theta(rb- \epsilon)^{-1}\gamma(x) \right)(1-V(x)) >(1-U(x)) \times
$$
$$
\left[\lambda^2r^{-1}U(x)\left\{1-2U(x) + \theta P(U(x)) \right\} +  r^{-2}\left(\frac{rb- \epsilon}{6} -\epsilon\right)(1- 16 \theta) \right]>0. 
$$
Hence, we constructed $C^3$- smooth functions $U(x), V(x)$ with positive derivatives on $\R_+$, satisfying the boundary conditions $(U, V)(0)= \alpha_2, \ (U, V)(+\infty)= \beta$ and the inequalities 
$$
U''(x)+U(x)(1-r-U(x)+rV(x)) >0, \ V''(x) +(bU(x)-\epsilon V(x))(1-V(x)) >0, \ x \geq 0. 
$$
Such functions exist for all sufficiently small $\epsilon$ (an explicit upper estimate for $\epsilon$ is presented above in this  appendix). Let $\eta: \R \to [0,1]$ be a smooth non-decreasing function such that $\eta(x) =0$ for all $x \leq x_0$,  $\eta(x) =1$ for $x \geq x_0+1$, where we choose $x_0\gg 1$ such that {$bU(x)-\epsilon V(x) \geq (1-V(x))br^2+0.5b$, $U(x) >r/(2r-1)$ for all $x \geq x_0$. Set
$$
U(t,x)= U(x+ct)- \mu r^2 \eta(x+ct)e^{-\delta t}, \  V(t,x)=  V(x+ct)- \mu  \eta(x+ct) e^{-\delta t},
$$
where  all new parameters are positive real numbers,  $\mu, \delta$ and $c$ are close to $0$. 
Then in the strip $0 \leq x+ct \leq x_0$, we clearly have that 
$$
U_{xx}(t,x) + U(t,x)(1-r-U(t,x)+rV(t-h,x)) - U_t(t,x)= \sigma_1:= 
$$
$$
U''(x+ct) + U(x+ct)(1-r-U(x+ct)+rV(x+ct -ch)) - cU'(x+ct) >0;
$$
$$
V_{xx}(t,x) +(bU(t,x)-\epsilon V(t,x))(1-V(t,x)) - V_t(t,x) = \sigma_4:=
$$
$$
V''(x+ct) + (bU(x +ct)- \epsilon V(x+ct))(1-V(x+ct)) - cV'(x+ct) >0$$
for all sufficiently small $c>0$. Let now $x+ct \geq x_0$. Then, for all sufficiently small positive $\mu, \delta$ and $c$,  
$$
U_{xx}(t,x) + U(t,x)(1-r-U(t,x)+rV(t-h,x)) - U_t(t,x)\geq  $$
$$\sigma_1+ r^2\mu(\sigma_2+ \sigma_3\eta(x+ct))e^{-\delta t}>0, 
$$
where $\sigma_1$ is defined above and 
$$
\sigma_2= -\eta''(x+ct)+c\eta'(x+ct), \quad
\sigma_3= 
r(1- V(x+ct))+U(x+ct)(2- r^{-1}e^{\delta h}) - 1-\delta +$$
 $$r\mu \eta(x+ct-ch)e^{-\delta (t-h)} - \mu r^2 \eta(x+ct)e^{-\delta t} >0.
$$
Note that, for all $\xi= x+ct \geq 0$, and appropriate $\xi_c \in [\xi -ch, \xi]$, 
$$
\sigma_1= U''(\xi) + U(\xi)(1-r-U(\xi)+rV(\xi)) +rU(\xi)(V(\xi -ch)-V(\xi)) - cU'(\xi)>
$$
$$
 U(\xi)(\lambda^2(1-2U(\xi))(1-U(\xi)) +$$ 
 $$
  \theta \gamma (\xi)- c \lambda(1-U(\xi))-ch\lambda  U(\xi_c)(1-U(\xi_c))(1+ \theta Q(U(\xi_c))) >
$$
$$
 U(\xi)\left(\lambda^2(1-2U(\xi))(1-U(\xi)) + \theta \gamma (\xi)- (2h+1)c\lambda (1-U(\xi))\right) >0  
$$
if $c>0$ is small enough. 
Finally, if $x+ct \geq x_0$, then 
$$
V_{xx}(t,x) +(bU(t,x)-\epsilon V(t,x))(1-V(t,x)) - V_t(t,x) =\sigma_4+ \mu(\sigma_2+ \sigma_5\eta(x+ct))e^{-\delta t}>0
$$
where 
$$
\sigma_5= (bU(x+ct)-\epsilon V(x+ct)) + (1-V(x+ct)+\mu\eta(x+ct)e^{-\delta t})(\epsilon  -br^2)- \delta>0. 
$$
%%%%%%%%%%%%%%%%%%%%%%%%%%%%%%%%%%%%%%%%%%%%%%%%%%%%%%%%%%%%%%%%%%%%%%%%%%%%%%%%%%%%%%%%%%%%%%%%%

\section{Construction of a  super-solution}

Let  impair integer $k > r/(r-1)$ be such that $k\alpha_{22}>1$. Assume that $\epsilon \in (0, rb/2)\cap (0,1)$. For small $\lambda \in (0, \epsilon/(k+1)^8)$, we will consider the positive increasing solution $v(x) = v(x, \lambda)$ of the boundary value problem 
$$
v'(x)= \lambda v(x)(1-v(x))^k, \ v(-\infty) = 0, \ v(+\infty) = 1, 
$$ 
normalised by the condition $v(0) = \alpha_{22}$. Then $v(x^*)= 1/k$ and $v(x_*)= 1/(k+1)$ for some $x_*<x^*=x^*(\lambda) <0$. 
Note that $x^*(\lambda) \to -\infty$ as $\lambda \to 0^+$, 
$$
v''(x)= \lambda ^2v(x)(1-v(x))^{2k-1}(1-(k+1)v(x)).  
$$ 
The smooth function $u(x)$ will be defined from the linear equation
$$
bu(x)-\epsilon v(x)= \lambda \gamma(x) \leq 0, 
$$
where  $\gamma(x)=0$ for $x \geq x^*$, $\gamma(x)=-\epsilon v(x)(v(x)-1/k)^4=: - \epsilon P_5(v(x))$ for $x \leq x^*$. We observe that 
$
u(0) = \alpha_{21}
$
and, for the polynomial $P_5(v)= v(v-1/k)^4$ , 
$$
u(x)= \frac{\epsilon}{b}\left( v(x)- \lambda P_5(v(x)) \right), \quad u'(x)= \frac{\epsilon}{b}v'(x)\left( 1- \lambda P_5'(v(x)) \right),\quad x \leq x^*,
$$
so that $u'(x) >0$ for all $x \leq 0$ if $\lambda>0$ is sufficiently small. 

Now, for $x \in [x^*,0]$, the inequality 
\begin{equation}\label{fv}
v''(x)+ (bu(x)-\epsilon v(x))(1-v(x)) <0
\end{equation}
amounts to 
$
\lambda ^2v(x)(1-v(x))^{2k-1}(1-(k+1)v(x)) <0 
$
and therefore it is satisfied. 
On the other hand, for $x \leq x^*$, inequality (\ref{fv}) is equivalent to 
$$
\lambda(1-v(x))^{2(k-1)}(1-(k+1)v(x)) <\epsilon\left(v(x)-\frac 1 k\right)^4, 
$$
which is obvious for $x \in [x_*,x^*]$ and also holds for $x \leq x_*$ since 
$$
\lambda (1-v(x))^{2(k-1)}(1-(k+1)v(x)) <\lambda  < \epsilon/(k+1)^8 < \epsilon\left(v(x)-\frac 1 k\right)^4, \ x \leq x_*. 
$$
Next,  for $x \in [x^*,0]$,  the inequality 
\begin{equation}\label{uf}
u''(x)+u(x)(1-r-u(x)+rv(x)) <0
\end{equation}
holds true since, after a simplification, we find that, for $x \in [x^*,0]$,
$$
\lambda ^2(1-v(x))^{2k-1}(1-(k+1)v(x)) + 1-r +(r- \epsilon/b)v(x) \leq 
$$
$$
\lambda ^2(1-v(x))^{2k-1}(1-(k+1)v(x)) + 1-r +(r- \epsilon/b)v(0) = $$
$$\lambda ^2(1-v(x))^{2k-1}(1-(k+1)v(x)) <0. 
$$

If we take $x \leq x^*$, then (\ref{uf}) is true because, for all sufficiently small $\lambda >0$ and some real polynomial  $P_{2k+4}(v)$ of degree $2k+4$, it holds 
$$
v''(x) + v(x)(1-r+(r-\epsilon/b)v(x))+
$$
$$
\lambda\left[-P_5(v(x))'' -P_5(v(x))\left(1-r +\left(r-\frac{2\epsilon}{b}\right)v(x)+\frac{\lambda\epsilon}{b}P_5(v(x))\right)\right] < 
$$
$$
v(x)(1-r+\frac{r-\epsilon/b}{k} + \lambda ^2)+
\lambda\left[-P_5(v(x))'' -P_5(v(x))\left(1-r\right)\right] <
$$
$$
v(x)\left(1-r+\frac{r-\epsilon/b}{k} + \lambda ^2+ \lambda \left[\lambda^2P_{2k+4}(v(x))+ \frac{r-1}{ k^4}\right]\right) <0. 
$$
Finally, let $\eta: \R \to [0,1]$ be a smooth non-decreasing function such that $\eta(x) =0$ for all $x \geq x_0$,  $\eta(x) =1$ for $x \leq x_0-1$, where we choose $x_0\ll x_*$. Set
$$
U(t,x)= u(x)+ \mu (\epsilon/3b)\eta(x)e^{-\delta t}, \  V(t,x)=  v(x)+ \mu  \eta(x) e^{-\delta t},
$$
where  all new parameters are positive real numbers,  $\mu$ is close to $0$. 
Then in the strip $x_0 \leq x \leq 0$, we  have that $U(t,x)= u(x), \ V(t,x)= v(x)$ and therefore 
$$
\mathcal D_1(U,V):= U_{xx}(t,x) + U(t,x)(1-r-U(t,x)+rV(t-h,x)) - U_t(t,x) <0, 
$$
$$
\mathcal D_2(U,V): = V_{xx}(t,x) +(bU(t,x)-\epsilon V(t,x))(1-V(t,x)) - V_t(t,x) <0. 
$$
Now, if $x \leq x_0$, then for all sufficiently small $\mu >0$, $\delta >0$, 
$$
\mathcal D_1(U,V)= \mathcal D_1(u,v) +  \mathcal D_{1}(\mu (\epsilon/3b)\eta(x)e^{-\delta t}, \mu  \eta(x) e^{-\delta t})+
$$
$$
\mu\eta(x)u(x)e^{-\delta t} (-2\epsilon/3b+ r e^{\delta h})+ \mu (r\epsilon/3b)\eta(x)e^{-\delta t}v(x)<0, 
$$
$$
\mathcal D_2(U,V) =  v''(x) + \mu \eta''(x)e^{-\delta t}  + (\lambda \gamma(x)-(2/3)\epsilon\mu\eta(x) e^{-\delta t})(1- v(x)- \mu  \eta(x) e^{-\delta t}) + $$
$$
\mu\delta  \eta(x) e^{-\delta t}<
v''(x) +  \lambda \gamma(x)(1- v(x))+  $$
$$\mu e^{-\delta t}\left[ \eta''(x) -\eta(x)\epsilon\left\{\frac 23 \left(1- \frac{1}{k+1}- \mu\right) -\frac{\lambda}{k^5}\right\}+\delta \eta(x)\right] <0. 
$$

\bibliographystyle{siamplain}
\bibliography{BZ-bistable-arxiv}
%\printbibliography

\end{document}